\documentclass[10pt]{article}
\usepackage[english]{babel}
\usepackage{amsmath}
\usepackage{amsfonts}
\usepackage{amssymb}
\usepackage{amsthm}
\usepackage{amscd}
\usepackage{fullpage}
\usepackage{mathtools}

\usepackage{hyperref}

\newcommand{\Z}{\mathbb{Z}}
\newcommand{\R}{\mathbb{R}}

\newcommand{\con}{\nabla}
\newcommand{\C}{\mathbb{C}}

\newcommand{\E}{\mathbb{E}}

\newcommand{\mc}{\mathcal}
\newcommand{\mb}{\mathbb}

\newcommand{\eps}{\varepsilon}

\DeclareMathOperator{\Var}{Var}

\newcommand{\Pro}{\mathbb{P}}

\newcommand{\norme}[1]{\left \| #1 \right \|}

\newcommand{\abs}[1]{\left\lvert#1\right\rvert}

\newcommand{\vertiii}[1]{{\left\vert\kern-0.25ex\left\vert\kern-0.25ex\left\vert #1 
    \right\vert\kern-0.25ex\right\vert\kern-0.25ex\right\vert}}

\DeclareMathOperator{\diam}{diam}

\DeclareMathOperator{\osc}{osc}

\newtheorem{thm}{Theorem}

\newtheorem{Prop}[thm]{Proposition}
\newtheorem{Lem}[thm]{Lemma}
\newtheorem{Cor}[thm]{Corollary}

\newtheorem{Rem}[thm]{Remark}

\def\XXint#1#2#3{{\setbox0=\hbox{$#1{#2#3}{\int}$ }
\vcenter{\hbox{$#2#3$ }}\kern-.6\wd0}}

\title{Tightness of Liouville first passage percolation for $\gamma \in (0,2)$}
\author{Jian Ding \thanks{Department of Statistics, The Wharton School, University of Pennsylvania,  Philadelphia, PA 19104, USA. Partially supported by NSF grant DMS-1757479 and an Alfred Sloan fellowship.}  \and Julien Dub\'edat \thanks{Department of Mathematics, Columbia University, New York, NY 10027, USA. Partially supported by NSF grant DMS-1512853.} \and Alexander Dunlap \thanks{Department of Mathematics, Stanford University, Stanford, CA 94305, USA. Partially supported by an NSF Graduate Research Fellowship. Current address: Department of Mathematics, Courant Institute of Mathematical Sciences, New York University, New York, NY 10012, USA.} \and Hugo Falconet \thanks{Department of Mathematics, Columbia University, New York, NY 10027, USA.}}

\begin{document}

\maketitle

\begin{abstract}
We study Liouville first passage percolation metrics associated to a Gaussian free field $h$ mollified by the two-dimensional heat kernel $p_t$ in the bulk, and related star-scale invariant metrics. For $\gamma \in (0,2)$ and $\xi = \frac{\gamma}{d_{\gamma}}$, where $d_{\gamma}$ is the Liouville quantum gravity dimension defined in \cite{DG18}, we show that renormalized metrics $(\lambda_t^{-1} e^{
\xi p_t * h} ds)_{t \in (0,1)}$ are tight with respect to the uniform topology. We also show that subsequential limits are bi-H\"older with respect to the Euclidean metric, obtain tail estimates for side-to-side distances, and derive error bounds for the normalizing constants $\lambda_t$.
 \end{abstract}

\tableofcontents

\section{Introduction and main statement}

We consider the problem of rigorously constructing a metric for \emph{Liouville quantum gravity} (LQG), a random geometry formally given by reweighting Euclidean space by $e^{\gamma h}$, where $h$ is a Gaussian free field. LQG was originally introduced in the physics literature by Polyakov in 1981 \cite{P81}. In its mathematical form, the LQG measure is a special case of \emph{Gaussian multiplicative chaos}, introduced in \cite{Kahane85}.  In the last two decades, there has been an explosion of interest in the probability community towards rigorously constructing the relevant objects. In particular, the LQG measure was constructed rigorously in the regime $\gamma\le 2$, via a renormalization procedure, in \cite{DS11}. Other relevant work in this area includes \cite{RoV10,RV11,RV14,Shamov16,Berestycki17,RV17,APS18}.

Much remains open regarding the construction of the LQG metric. When $\gamma = \sqrt{8/3}$, LQG is intimately connected with the Brownian map \cite{LeGall10,LeGall13,Miermont13} and a metric for LQG has been constructed in \cite{MS15b,MS16a,MS16b}. Substantial work has also been devoted to understanding the distance exponents for natural discrete LQG metrics; see \cite{DZ16,DGo18,GHS16,GHS17,DG18,A19}. In \cite{DZ15,DZZ17} some non-universality results were established for first-passage percolation distance exponents for metrics of the form $e^{\gamma \phi_\delta}ds$, where $\phi_\delta$ is discretization of a log-correlated Gaussian field. This indicates that precisely understanding such exponents must involve rather fine information about the structure of the particular field in question.

The present study concerns the tightness of Liouville first-passage percolation (LFPP) metrics, which are natural smoothed LQG metrics. This proves the existence of subsequential limiting metrics. Given this, it remains to show that such limiting metrics are unique in law for each $\gamma\in(0,2)$ in order to complete the construction of the LQG metric in this regime. After this paper was posted, the latter task was carried out in the series of works \cite{GM19c,DFGPS19,GM19conf,GM19b,GM19existence}, thus completing the construction. The present study follows three main tightness results for discretized or smoothed LQG metrics. In \cite{DD16}, tightness of LFPP metrics (on a discrete lattice) was proved in the small noise regime for which $\gamma$ is very small. In \cite{DF18}, tightness was shown for metrics arising in the same way from $\star$-scale invariant fields, still in the small noise regime. In \cite{DD18}, tightness was shown for all $\gamma<2$ for the Liouville graph distance, which is a graph metric equal to the least number of Euclidean balls of a given LQG measure necessary to cover a path between a pair of points.

We consider a smoothed Gaussian field
\begin{equation}\label{eq:phideltadef}
\phi_{\delta}(x) := \sqrt{\pi} \int_{\delta^2}^1 \int_{\R^2} p_{\frac{t}{2}}(x-y) W(dy,dt)
\end{equation}
for $x\in\R^2$ and  $\delta \in (0,1)$, where $p_t(x-y) := \frac{1}{2\pi t}e^{-\frac{|x-y|^2}{2t}}$ and $W$ is a space-time white noise. This approximation is natural since it can be uniformly compared on a compact domain with a Gaussian free field $h$ mollified by the heat kernel defined on a slightly larger domain, viz. $\phi_{\sqrt{t}}$ and $p_{t/2} * h$ (where $*$ denotes the convolution operator) are comparable. Furthermore, this approximation provides some nice invariance and scaling properties on the full plane.

For $\gamma \in (0,2)$, we will use the notation 
\begin{equation}
\label{def:Exponent}
\xi := \gamma/d_{\gamma}
\end{equation} where $d_{\gamma}$ is the ``Liouville quantum gravity dimension" defined in \cite{DG18}. It is known (see Theorem 1.2 and Proposition 1.7 in \cite{DG18}) that the function $\gamma \mapsto \gamma/d_{\gamma}$ is strictly increasing and continuous on $(0,2)$. Therefore, in this article we will be interested in the range $\xi \in (0,(2/d_2)^-)$, where $(2/d_2)^-=\lim_{\gamma\uparrow 2}\gamma/d_\gamma$. 

We consider the length metric  $e^{\xi \phi_{\delta}} ds$ (equivalently, the metric whose Riemannian metric tensor is given by $e^{2 \xi \phi_{\delta}} ds^2$), restricted to the unit square $[0,1]^2$. We recall that a length metric is a metric such that the distance between two points is given by the infimum over the arc lengths of paths connecting the two points. We denote by $\lambda_{\delta}$ the median of the left-right distance of $[0,1]^2$ for the metric $e^{\xi \phi_{\delta}} ds$. Our main theorem is the following.

\begin{thm}
\label{thm:MainTheorem}
\begin{enumerate}\item If $\gamma \in (0,2)$, then $\left( \lambda_{\delta}^{-1} e^{\xi \phi_{\delta}} ds \right)_{\delta \in (0,1)}$ is tight with respect to the uniform topology on the space of continuous functions $[0,1]^2 \times [0,1]^2\to\R^+$. Furthermore, any subsequential limit is almost surely bi-H\"older with respect to the Euclidean metric on $[0,1]^2$. 
\item Let $K=[0,1]^2$. If $h$ is a Gaussian free field with zero boundary conditions on a bounded open domain $D$ containing $K$ (extended to zero outside of $D$), then the internal metrics  $(\lambda_{\sqrt{\delta}}^{-1} e^{\xi p_{\frac{\delta}{2}} * h} ds)_{\delta \in (0,1)}$ on $K$ are tight with respect to the uniform topology of continuous functions $K\times K\to \R^+$.
\end{enumerate}
Furthermore, the normalizing constants $(\lambda_{\delta})_{\delta \in (0,1)}$ satisfy
\begin{equation}
\label{eq:BornesExpo}
\lambda_{\delta} = \delta^{1-\xi Q} e^{O\left( \sqrt{| \log \delta |} \right)}
\end{equation}
where $Q = \frac{2}{\gamma} + \frac{\gamma}{2}$. 
\end{thm}

A year after our article was posted, the subsequent work \cite{DG20} proved a similar result to ours when $\xi\ge (2/d_2)^-$. However, in that case the tightness does not hold in the uniform topology and the Beer topology on lower semicontinuous functions was used.

In order to establish the tightness of the family of renormalized metrics $ (d_{\phi_{\delta}})_{\delta \in (0,1)} := (\lambda_{\delta}^{-1} e^{\xi 
\phi_{\delta}} ds)_{\delta \in (0,1)}$, we prove a number of uniform estimates for that family (which also hold when the approximation is the GFF mollified by the heat kernel). Such estimates that are closed under weak convergence also apply to subsequential limits. Let us summarize these properties. Let ${\mathcal D}$ denote the family of laws of $d_{\phi_{\delta}}$, $\delta\in (0,1)$ (i.e.\ seen as random continuous functions on $([0,1]^2)^2$), and $\overline{\mathcal D}$ denotes its closure under weak convergence (i.e., $\overline{\mathcal D}$ also includes the laws of all subsequential limits).  

\begin{enumerate}
\item Under any $\Pro\in\overline{\mathcal D}$, $d$ is $\Pro$-a.s.\ a length metric. This is clear for the renormalized metrics $d_{\phi_\delta}$ by definition, and the property of being a length metric extends to limits. (See \cite[Exercise~2.4.19]{BBI}.)

\item 
If $d$ is a metric on $\R^2$ and $R$ is a rectangle, we denote by $d(R)$ the left-right length of $R$ for $d$. We have the following tail estimates. There exists $c, C > 0$ such that for $s > 2$, uniformly in $\Pro\in\overline{\mathcal D}$ we have
\begin{align}
& c e^{-C s^2} \leq  \Pro \left( d(R)  \leq e^{-s} \right) \leq C e^{-c s^2},\label{eq:lowertailbd} \\
& c e^{-C s^2} \leq \Pro \left( d(R) \geq e^s \right) \leq C e^{-c \frac{s^2}{\log s}}.\label{eq:uppertailbd}
\end{align}
The upper bounds are proved in Section~\ref{Sec:TailEstimates}, while the lower bounds are consequences of the Cameron--Martin theorem, considering shifts of the field at the coarsest scale as in \cite[Section~5.4]{DF18}. 

\item
If $d$ is a metric on $\R^2$ and $R$ is a rectangle, we denote by $\mathrm{Diam}(R,d)$ the diameter of $R$ for $d$. We have the following uniform first moment bound:
\begin{equation}
\sup_{\Pro\in\overline{\mathcal D}} \E \left( \mathrm{Diam}(R,d) \right) < \infty.
\end{equation}
This is shown in the course of the proof of Proposition~\ref{pro:diam} below.

\item
Under any $\Pro\in\overline{\mathcal D}$, $d$ is $\Pro$-a.s. bi-H\"older with respect to the Euclidean metric and we have the following bounds for exponents: for $\alpha < \xi (Q-2)$, $\beta > \xi (Q+2)$, and $R$ a rectangle, the families
\begin{equation}
\left( \sup_{x,x' \in R} \frac{|x-x'|^{\alpha}}{d(x,x')}  \right)_{{\mathcal L}(d)\in\overline{\mathcal D}} \quad \text{and}  \quad
\left( \sup_{x,x' \in R} \frac{d(x,x')}{|x-x'|^{\beta}} \right)_{{\mathcal L}(d)\in\overline{\mathcal D}}
\end{equation}
are tight. Here $\mathcal{L}(d)$ means the law of $d$. These properties are shown in Proposition~\ref{Prop:Holder} below.
\end{enumerate}

Let us also mention that subsequential limits are consistent with the Weyl scaling: for a function $f$ in the Cameron-Martin space of the Gaussian free field $h$, for any coupling $(h,d)$ associated to a subsequential limit of the sequence of laws of $((h, \lambda_{\sqrt{\delta}}^{-1} e^{\xi p_{\frac{\delta}{2}} * h} ds))_{\delta>0}$, the couplings $(h,d)$ and $(h+f,e^{\xi f} \cdot d)$ are mutually absolutely continuous with respect to each other and the associated Radon-Nikod\'ym derivative is the one of the first marginal. This can be proved using similar arguments to those of \cite[Section~7]{DF18}. An analogue of this property for the Liouville measure together with the  conservation of the Liouville volume average is enough to characterize the Liouville measure, as seen by Shamov in \cite{Shamov16}.

It may be interesting to draw a parallel between our work and those in random planar maps, since both aim to obtain the scaling limits of random metrics and the limiting objects are related for $\gamma = \sqrt{8/3}$. We start with a brief overview on the convergence of random planar maps. Chassaing and Schaeffer \cite{CS04} identified $n^{1/4}$  as the proper scaling and compute certain limiting functionals for random quadrangulations. Marckert and Mokkadem \cite{MM06} established limit theorems (in a sense weaker than Gromov-Hausdorff) and introduced the Brownian map. Le Gall \cite{LeGall07} showed tightness for rescaled $2p$-angulations in the Gromov-Hausdorff topology and shows that the limiting topology is the same as the Brownian map. Le Gall and Paulin \cite{LeGallPaulin} showed that the limiting topology is that of the 2-sphere and Le Gall studied properties of geodesics in \cite{LeGall10}.  Finally, Le Gall \cite{LeGall13} (resp. Miermont \cite{Miermont13}) proved the uniqueness of subsequential limits of uniform triangulations and $2p$-angulations (resp. quadrangulations). In both cases, the proof relied on a careful study of geodesics and in particular on confluence properties, together with a rough bound quantifying an approximate equivalence of the two metrics to match.

In our framework, an important result was obtained in \cite{DG18}, where the authors identified the exponent of LFPP distances to be $1 -\xi Q + o(1)$. By contrast, in the random planar map setting, the normalization is exactly $n^{1/4}$ and Chassaing and Schaeffer \cite{CS04} obtained the convergence in law of some observables. In our case, the tightness of any observable renormalized by its median is far from obvious. This is a common feature of some concentration problems for extrema of random fields. As an analogy, one can consider the problem of the tightness of the maxima of branching random walks (BRW), where, also by subadditivity, the expected value of the maxima of BRW on a $d$-ary tree at level $n$ is of order $(x^{*} + o(1))n$ for some $x^{*}$ (which depends on the rate function of the distribution of the increments). A powerful and well-understood method in proving tightness for BRW is by an explicit truncated second moment estimate which computes the expected maxima up to additive $O(1)$ constant (see \cite{BDZbrw} for maxima of BRW and \cite{BDZgff} for maxima of discrete GFF). In contrast, in our setup, explicit computation on the distance seems really difficult; in fact it remains a major challenge to compute the value of the distance exponent, let alone computing the distance up to constant. In order to circumvent this difficulty, we had to build our proof by exploring delicate intrinsic structure of the distance. We point out here that it was shown in \cite{A19} that for $\gamma \in (0,2)$, $1-\xi Q \geq 0$ (and it is believed to be $>0$), therefore  the normalization $\lambda_{\delta}$ should be thought as small.

Furthermore, in our setting where the metrics are on a compact subset of $\C$, we can directly use the uniform topology instead of working with the Gromov-Hausdorff topology (note that the former is stronger than the latter). In this paper, we show tightness for the full subcritical range $\gamma \in (0,2)$ of renormalized side-to-side crossing lengths, point-to-point distance and metrics. Limiting metrics are bi-H\"older with respect to the Euclidean metric.

\subsection{Strategy of the proof and comparison with previous works}

In contrast with previous works on the LQG measure, the variational  problem defining the LQG metric means that most direct computations are impossible, and in particular most of techniques used in the theory of Gaussian multiplicative chaos and LQG measure are unavailable. This necessitates the more intricate multiscale geometric arguments that we employ.

Our tightness proof relies on two key ingredients, a Russo-Seymour-Welsh argument and multiscale analysis. In both parts we extend and refine many arguments used in the previous works \cite{DD16,DF18,DD18} on the tightness of various types of LQG metrics. 

\paragraph{Russo-Seymour-Welsh.}

The RSW argument relates, to within a constant factor, quantiles of the left--right LFPP crossing distances of a ``portrait'' rectangle and of a ``landscape'' rectangle. (By a crossing distance we simply mean the distance between two opposite sides of a rectangle.) In \cite{DD16,DD18}, these crossings are referred to as ``easy'' and ``hard'' respectively. The utility of such a result is that crossings of larger rectangles necessarily induce easy crossings of subrectangles, while hard crossings of smaller rectangles can be glued together to create crossings of larger rectangles. Thus, multiscale analysis arguments can establish lower bounds in terms of easy crossings and upper bounds in terms of hard crossings. RSW arguments then allow these bounds to be compared.

RSW arguments originated in the works \cite{Russo78, SW78, Russo81} for Bernoulli percolation, and have since been adapted to many percolation settings. The work \cite{DD16} introduced an RSW result for LFPP  in the small noise regime based on an RSW result for Voronoi percolation devised by Tassion \cite{Tassion16}. Tassion's result is beautiful but intricate, and becomes quite complex when it is adapted to take into account the weights of crossing in the first-passage percolation setting, as was done in \cite{DD16}.

The RSW approach of this paper is based on the much simpler approach introduced in \cite{DF18}, which relies on an approximate conformal invariance of the field. (We recall that the Gaussian free field is exactly conformally invariant in dimension $2$, and that the LQG measure enjoys an exact conformal covariance.) Roughly speaking, the conformal invariance argument relies on writing down a conformal map between the portrait and landscape rectangles, and analyzing the effect of such a map on crossings of the rectangle. We note that the approximate conformal invariance used in this paper relies in an important way on the exact independence of different ``scales'' of the field, which is manifest in the independence of the white noise at different times in the expression \eqref{eq:phideltadef}. Thus, the argument we use here is not immediately applicable to mollifications of the Gaussian free field by general mollifiers (for example, the common ``circle-average approximation'' of the GFF). The RSW argument of \cite{DF18} was also adapted in \cite{DD18} to the Liouville graph distance case. 

\paragraph{Tail estimates.}

Once the RSW result is established, we derive tail estimates with respect to fixed quantiles. The lower tail estimate is unconditional, while the upper tail estimate depends on a quantity $\Lambda_n$ measuring the concentration at the current scale, which will later be uniformly bounded by an inductive argument. 

\paragraph{Multiscale analysis.}

With RSW and tail estimates in hand, we turn to the multiscale analysis part of the paper. This argument turns on the Condition~(T) formulated in \eqref{eq:AssA} below, which, informally, states that the arclength of the crossing is not concentrated on a small number of subarcs of small Euclidean diameter. The argument of \cite{DD18} requires similar input, which is a key role of the subcriticality $\gamma<2$. While \cite{DD18} relies directly on certain scaling symmetries of the Liouville graph distance to use subcriticality, the present work relies on the characterization of the Hausdorff dimension $d_\gamma$ obtained in \cite{DG18}, along with some weak multiplicativity arguments and concentration obtained from percolation arguments.

\paragraph{Condition (T).}

Our formulation of Condition~(T), which has not appeared in previous works, precisely captures the property of the metric needed to obtain the tightness of the left--right crossing distances, the existence of the exponent, and the tail estimates (via a uniform bound on the $\Lambda_n$).

Condition~(T) makes sense for LFPP with any underlying field and any parameter $\xi$. In particular, this condition or a variant thereof could possibly hold for LFPP for some $\xi >2/d_2$. 
Therefore, a byproduct of the present work is a simple criterion (that implies, as noted above, tightness of the crossing distances, existence of exponents, and tail estimates)  that may be applicable more generally.

The utility of Condition~(T) is that it allows us to use an Efron--Stein argument to obtain a contraction in an inductive bound on the crossing distance logarithm variance. Informally, since the crossing distance feels the effect of many different subboxes, the subbox crossing distances are effectively being averaged to form the overall crossing distance. This yields a contraction in variance. (Of course, the coarse scales also contribute to the variance, and hence the variance of the crossing distance does not decrease as the discretization scale decreases but rather stays bounded.) 

The way we verify Condition (T) is quite rough: we bound the field uniformly over a coarse grained geodesic by the supremum of the field over the unit square. It turns out that this bound together with the identification of the exponent $1-\xi Q$ is enough to establish the condition.

\paragraph{Tightness of the metrics.}

Once the tightness of the left--right crossing distance is established, we turn to the tightness of the diameter and of the metric itself. This is done by a chaining argument, and requires again $\xi<2/d_2$. The diameter is not expected to be tight when $\xi>2/d_2$, since there are points that become infinitely distant from the bulk of the space as the discretization scale goes to $0$.

\section{Description and comparison of approximations}
\label{Sec:Convolution}
We recall that
a white noise $W$ on $\mb{R}^d$ is a random Schwartz distribution such that for every smooth and compactly supported test function $f$, $\langle W, f \rangle$ is a centered Gaussian variable with variance $\| f \|_{L^2(\mb{R}^d)}$ (see e.g. \cite{DaPrato}). The main approximation of the Gaussian free field that we consider in this paper is defined for $\delta \in (0,1)$ by
\begin{equation}
\label{Def:PhiDelta}
\phi_{\delta}(x) := \sqrt{\pi} \int_{\delta^2}^1 \int_{\R^2} p_{\frac{t}{2}}(x-y) W(dy,dt) 
\end{equation}
where $p_t(x-y) := \frac{1}{2\pi t}e^{-\frac{|x-y|^2}{2t}}$ and $W$ is a space-time white noise on $[0,1] \times \R^2$. This approximation is different than the one considered in \cite{DF18}
which is
$$
\tilde{\phi}_{\delta}(x) := \int_{\delta}^1 \int_{\R^2} k \left( \frac{x-y}{t} \right) t^{-3/2} W (dy, dt) 
$$
for a smooth nonnegative bump function $k$, radially symmetric and with compact support. Up to a change of variable in $t$, the difference is essentially replacing $p_1$ by $k$. Both fields are normalized in such a way that $\E ( \phi_{0}(x) \phi_{0}(y) ) = - \log |x-y| + g(x,y)$ with $g$ continuous (see e.g. Section 2 in \cite{DF18}): this is the reason for the factor $\sqrt{\pi}$ in \eqref{Def:PhiDelta}.

Let us mention that $\star$-scale invariant Gaussian fields with compactly-supported bump function $k$ 
\begin{enumerate}
\item
\label{item:a}
are invariant under Euclidean isometries,

\item
\label{item:b}
have finite-range correlation at each scale,

\item
\label{item:c}
and have
convenient scaling properties.
\end{enumerate}
The Gaussian field $\phi_{\delta}$ introduced above satisfies \ref{item:a} and \ref{item:c} but not \ref{item:b}. Because of the lack of finite-range correlation, we will also use a field $\psi_{\delta}$ (defined in the next section) which satisfies \ref{item:a} and \ref{item:b} such that $\sup_{n \geq 0} \norme{ \phi_{0,n} - \psi_{0,n}}_{L^{\infty}([0,1]^2)}$ has Gaussian tails, where we use the notation $\phi_{0,n}$ for $\phi_{\delta}$ with $\delta = 2^{-n}$. 

\subsection{Basic properties of $\phi_{\delta}$ and $\psi_{\delta}$}

\paragraph{Scaling property of $\phi_{\delta}$.} We use the scale decomposition
$$
\phi := \sum_{n \geq 0} \phi_n ~~ \text{where} ~~ \phi_{n}(x)  = \sqrt{\pi} \int_{2^{-2(n+1)}}^{2^{-2n}} \int_{\R^2} p_{\frac{t}{2}}(x-y) W(dy,dt)
$$
If we denote by $C_n$ the covariance kernel of $\phi_n$, so $C_n(x,x') = \E (\phi_n(x) \phi_n(x'))$, then we have
$$
C_n(x,x') = \int_{2^{-2(n+1)}}^{2^{-2n}} \frac{1}{2 t} e^{-\frac{|x-x'|^2}{2t}} dt = C_0(2^{n}x, 2^n x').
$$
Therefore, the law of $(\phi_{n}(x))_{x \in [0,1]^2}$ is the same as $(\phi_{0}(2^n x))_{x \in [0,1]^2}$. Because of the  $\frac{1}{2t}$ above,  we choose $\delta^2$ and not $\delta$ in \eqref{Def:PhiDelta} so that the pointwise variance $\phi_{\delta}$ is $\log \delta^{-1}$. Similarly, for $0 < a < b$ and $x \in \R^2$, set \begin{equation}\phi_{a,b}(x) := \sqrt{\pi} \int_{a^2}^{b^2} \int_{\R^2} p_{\frac{t}{2}}(x-y) W(dy,dt)\label{eq:phiabdef}\end{equation} and note that we have the scaling identity $\phi_{a,b}(r\cdot) \overset{(d)}{=} \phi_{a/r,b/r}(\cdot)$. Indeed, we have
$$
\E (\phi_{a,b}(rx) \phi_{a,b}(rx')) = \pi \int_{a^2}^{b^2} \int_{\R^2} p_{\frac{t}{2}}(rx-y) p_{\frac{t}{2}}(y-rx') dy dt= \pi \int_{a^2}^{b^2} p_t(r(x-x')) dt = \int_{a^2}^{b^2} \frac{1}{2t} e^{- \frac{r^2 |x-x'|^2}{2t}} dt ,
$$
and by the change of variable $t = r^2 u $, this gives $$
\int_{a^2}^{b^2} \frac{1}{2t} e^{- \frac{r^2 |x-x'|^2}{2t}} dt  = \int_{(a/r)^2}^{(b/r)^2} \frac{1}{2t} e^{- \frac{ |x-x'|^2}{2t}} dt = \E (\phi_{a/r,b/r}(x) \phi_{a/r,b/r}(x')) .
$$ 
We will use the notation $\phi_{k,n}$ when $a = 2^{-n}$ and $b = 2^{-k}$ for $0 \leq k \leq n$.

\paragraph{Maximum and oscillation of $\phi_{\delta}$.} We have the same estimates for the supremum of the field $\phi_{0,n}$ as those for the $\star$-scale invariant case considered in \cite{DF18} (it is essentially a union bound combined with a scaling argument). The following proposition corresponds to Lemma 10.1 and Lemma 10.2 in \cite{DF18}.
\begin{Prop}[Maximum bounds] 
\label{Prop:MaxBound}
We have the following tail estimates for the supremum of $\phi_{0,n}$ over the unit square: for $a > 0$, $n \geq 0$,
\begin{equation}
\label{eq:MaxBoundTail}
\Pro \left( \max_{[0,1]^2} |\phi_{0,n}| \geq a(n + C \sqrt{n}) \right) \leq C 4^n e^{-\frac{a^2}{\log 4} n}
\end{equation}
as well as the following moment bound: if $\gamma < 2$, then
\begin{equation}
\label{eq:MaxBoundExp}
\E (e^{\gamma \max_{[0,1]^2} | \phi_{0,n} |}) \leq 4^{\gamma n + O(\sqrt{n})}
\end{equation}
\end{Prop}
We will also need some control on the oscillation of the field $\phi_{0,n}$. We introduce the following notation for the $L^{\infty}$-norm on a subset of $\R^d$. If $A$ is a subset of $\R^d$ and $f: A \to \R^m$, we set
\begin{equation}
\label{Def:Norme2}
\norme{f}_A := \sup_{x \in A} |f(x)|
\end{equation}
We introduce the following notation to describe the oscillation of a smooth field $\phi$: if $A \subset \R^2$ we set
\begin{equation}
\label{Def:Osc}
\osc_A(\phi) := \diam(A) \norme{ \con \phi }_A,
\end{equation}
so that if $A$ is convex then $\sup_{x,y \in A} | \phi(x) - \phi(y)| \leq \osc_A(\phi)$ and \[\max_{P \in \mc{P}_n, P \subset [0,1]^2} \osc_{P}(\phi_{0,n}) \leq C 2^{-n} \norme{\con \phi_{0,n}}_{[0,1]^2},\] where ${\mc P}_n$ denotes the set of dyadic blocks at scale $n$, viz. \begin{equation}\label{Pndef}\mc{P}_n := \lbrace 2^{-n}([i,i+1]\times[j,j+1]) : i,j \in \Z \rbrace.
\end{equation}
In order to simplify the notation $P \in \mc{P}_n, P \subset [0,1]^2$ later on, we also set 
\begin{equation}
\label{eq:PK1}
\mc{P}_n^1 := \lbrace P \in \mc{P}_n : P \subset [0,1]^2 \rbrace.
\end{equation}
\begin{Prop}[Oscillation bounds]
We have the following tail estimates for the oscillation of $\phi_{0,n}$: there exists $C> 0$, $\sigma^2 > 0$, so that, for all $x,\eps > 0$, $n \geq 0$, 
\begin{equation}
\Pro \left( 2^{-n} \norme{ \con \phi_{0,n} }_{[0,1]^2} \geq x \right) \leq C 4^n e^{-\frac{x^2}{2\sigma^2}}
\label{eq:OscBoundTail}
\end{equation}
as well as the following moment bound: for $a >0$, there exists $c_a > 0$ so that for $n \geq 0$,
\begin{equation}
\label{eq:OscBoundExp}
\E \left( e^{a n^{\eps} 2^{-n} \norme{ \con \phi_{0,n} }_{[0,1]^2} }\right) \leq e^{c_a n^{\frac{1}{2}+\eps} + O(n^{2\eps})}
\end{equation}
\end{Prop}

\begin{proof}
Inequality
\eqref{eq:OscBoundTail} was obtained between Equation (10.3) and Equation (10.4) in \cite{DF18}. Now, we prove \eqref{eq:OscBoundExp}.
Set $a_n := a n^{\eps}$, $O_n = 2^{-n} \norme{ \con \phi_{0,n}}_{[0,1]^2}$, and take $x_n = a_n \sigma^2 + \alpha \sigma \sqrt{n}$ with $\alpha > 0$ so that $\frac{\alpha^2}{2} = \log 4$. We have, using \eqref{eq:OscBoundTail}, 
$$
\int_{e^{a_n x_n}}^{\infty} \Pro \left( e^{a_n O_n} \geq x \right)  dx = \int_{a_n x_n}^{\infty} \Pro \left( e^{a_n O_n} \geq e^s \right) e^s  ds \leq C 4^n \int^\infty_{a_n x_n} e^{-\frac{s^2}{2 a_n^2 \sigma^2}} e^s ds
$$
By a change of variable ($s \leftrightarrow a_n \sigma s + (a_n \sigma)^2$), we get 
\[\int_{a_n x_n}^\infty e^{-\frac{s^2}{2 a_n^2 \sigma^2}} e^s ds =  a_n \sigma  e^{\frac{1}{2} a_n^2 \sigma^2} \int_{\frac{x_n}{\sigma} - a_n \sigma}^{
\infty} e^{-\frac{s^2}{2}} ds = a_n \sigma e^{\frac{1}{2} a_n^2 \sigma^2} \int_{\alpha \sqrt{n}}^{\infty} e^{-\frac{s^2}{2}} ds\] since $x_n = a_n \sigma^2 + \alpha \sigma \sqrt{n}$. Using that $\int_{a}^{\infty} e^{-b x^2} dx \leq (2ab)^{-1} e^{-ba^2}$, we get $\int_{e^{a_n x_n}}^{\infty} \Pro \left( e^{a_n O_n} \geq x \right)  dx \leq  e^{O(n^{2\eps})}$. The result follows from writing $\E (e^{a_n O_n}) \leq e^{a_n x_n} + \int_{e^{a_n x_n}}^{\infty} \Pro(e^{a_n O_n} \geq x) dx$.
\end{proof}

\paragraph{Definition of $\psi_{\delta}$.} We fix a smooth, nonnegative, radially symmetric  bump function $\Phi$ such that $0 \leq \Phi \leq 1$ and $\Phi$ is equal to one on $B(0,1)$ and to zero outside $B(0,2)$. We also fix small constants $r_0 > 0$ and $\eps_0 > 0$. We will specify these constants later on. In particular, $\eps_0$ appears in the main proof in \eqref{Def:PK} and its final effect is in \eqref{eq:EffectE0}. All other constants $C,c$ will implicitly depend on $r_0$ and $\eps_0$. Then, we introduce for each $\delta \in [0,1]$, the field
$$
\psi_{\delta}(x) := \int_{\delta^2}^1 \int_{\R^2} \Phi_{\sigma_t} (x-y) p_{\frac{t}{2}}(x-y)  W(dy,dt) = \int_{\delta^2}^1 \int_{\R^2}  p^{\mathrm{Tr}}_{\frac{t}{2}}(x-y)  W(dy,dt) 
$$
\begin{equation}
\label{Def:sigma}
\text{where} \quad \sigma_t = r_0 \sqrt{t} | \log t |^{\eps_0}, \quad \Phi_{\sigma_t}(\cdot) := \Phi(\cdot/\sigma_t) \quad \text{and} \quad p^{\mathrm{Tr}}_{\frac{t}{2}} := p_{\frac{t}{2}} \Phi_{\sigma_t}.
\end{equation}
Thanks to the truncation, the fields $(\psi_{\delta})_{\delta \in [0,1]}$ have finite correlation length $8 r_0 \sup_{t \in [0,1]} \sqrt{t} | \log t |^{\eps_0}$. 

\paragraph{Decomposition in scales and blocks of $\psi_{\delta}$.} We have the scale decomposition 
\begin{equation} 
\psi(x)  \coloneqq \int_{0}^1 \int_{\R^2}  p^{\mathrm{Tr}}_{\frac{t}{2}}(x-y)  W(dy,dt)
= \sum_{k = 1}^{\infty} \sum_{P \in \mc{P}_k}  \int_{2^{-2k}}^{2^{-2k + 2}} \int_{P}  p^{\mathrm{Tr}}_{\frac{t}{2}}(x-y)  W(dy,dt)  = \sum_{k \geq 1} \sum_{P \in \mc{P}_k} \psi_{k,P}(x)
\end{equation}
where $\psi_{k,P}$ is defined for $P \in \mc{P}_k$ by $\psi_{k,P}(x) :=  \int_{2^{-2k}}^{2^{-2k + 2}} \int_{P}  p^{\mathrm{Tr}}_{\frac{t}{2}}(x-y)  W(dy,dt)$ and thus has correlation length less than $C k^{\eps_0} 2^{-k}$. In particular, a fixed block field is only correlated with fewer than $C k^{2 \eps_0}$ other block fields at the same scale. In fact, when we apply the Efron-Stein inequality (see \eqref{eq:EfronStein}) we will use the following decomposition:
\begin{equation}
\label{eq:BlockDecompoPsi}
\psi_{0,n} = \psi_{0,K} + \sum_{P \in \mc{P}_K} \psi_{K,n,P}(x) \quad
\text{where} \quad \psi_{K,n,P}(x) :=  \int_{2^{-2n}}^{2^{-2K + 2}} \int_{P}  p^{\mathrm{Tr}}_{\frac{t}{2}}(x-y)  W(dy,dt).
\end{equation}
We note that there is a formal conflict in notation between \eqref{eq:phiabdef} and \eqref{eq:BlockDecompoPsi}, but it will always be clear from context whether the second subscript is a number or an element of $\mathcal{P}_k$ (a set), so confusion should not arise. 

\paragraph{Variance bounds for $\phi_{\delta}$ and $\psi_{\delta}$.} Later on we will need the following lemma.
\begin{Lem}
\label{lem:FieldScale}
There exists $C > 0$ so that for $\delta \in [0,1]$ and $x,x' \in \R^2$, we have
\begin{equation}
\label{eq:sndbound}
\Var \left( \phi_{\delta}(x) - \phi_{\delta}(x') \right) +\Var \left( \psi_{\delta}(x) - \psi_{\delta}(x') \right) \leq C \frac{| x - x'|}{\delta}.
\end{equation}
\end{Lem}

\begin{proof}
We start by estimating the first term. Using the inequality $1 - e^{-z} \leq z \leq \sqrt{z}$ for $z \in [0,1]$ and $1-e^{-z} \leq 1 \leq \sqrt{z}$ for $z \geq 1$ we get
\begin{multline*}
\Var \left( \phi_{\delta}(x) - \phi_{\delta}(x')   \right) = C \int_{\delta^2}^1 \left( p_{\frac{t}{2}} \ast p_{\frac{t}{2}} (0)  - p_{\frac{t}{2}} \ast p_{\frac{t}{2}} (x-x') \right) dt \\
  = C  \int_{\delta^2}^1 \left( p_t(0) - p_t(x-x') \right) dt = C \int_{\delta^2}^1 \frac{1}{t} (1 - e^{-\frac{|x-x'|^2}{2t}}) dt \leq C |x-x'| \int_{\delta^2}^1 \frac{dt}{t^{3/2}} = C \frac{|x-x'|}{\delta}.
\end{multline*}
Similarly, for the second term, we have
$$
\Var \left( \psi_{\delta}(x) - \psi_{\delta}(x')   \right)  =  C \int_{\delta^2}^1 \left( p^{\mathrm{Tr}}_{\frac{t}{2}} \ast p^{\mathrm{Tr}}_{\frac{t}{2}} (0)  - p^{\mathrm{Tr}}_{\frac{t}{2}} \ast p^{\mathrm{Tr}}_{\frac{t}{2}} (x-x') \right) dt.
$$
Set $p^{\mathrm{Tr}}_{\frac{t}{2}} \ast p^{\mathrm{Tr}}_{\frac{t}{2}} =: p_t(x) q_t(x)$. Using the identity $p_{t/2}(y) p_{t/2}(x-y) = p_t(x) p_{t/4}(y-x/2)$ we get
$$
q_t(x) = \int_{\R^2} \frac{p_{t/2}(y) p_{t/2}(x-y)}{p_t(x)} \Phi_{\sigma_t}(y) \Phi_{\sigma_t}(x-y) dy = \int_{\R^2} p_{t/4}(y-x/2) \Phi_{\sigma_t}(y) \Phi_{\sigma_t}(x-y) dy.
$$
We rewrite the variance in terms of $q_t$: replacing $x-x'$ by $z$ we look at
\begin{align*}
\Var \left( \psi_{\delta}(x) - \psi_{\delta}(x')   \right)  &= C \int_{\delta^2}^1 (p_t(0) q_t(0) - p_t(z)q_t(z))dt \\&= C \int_{\delta^2}^1p_{t}(0) (q_t(0) - q_t(z)) dt + C \int_{\delta^2}^1 q_t(z) (p_t(0) - p_t(z))dt.
\end{align*}
We deal with these two terms separately. For the second one, since $0 \leq \Phi \leq 1$, we have $0 \leq q_t \leq 1$. Therefore, following what we did for $\phi_{\delta}$ above we directly have $0 \leq \int_{\delta^2}^1 q_t(z) (p_t(0) - p_t(z))dt \leq C \frac{|z|}{\delta}$. For the first term, since $p_t(0) = C t^{-1}$, it is enough to get the bound $\sqrt{t} | q_t(0) - q_t(z) | \leq C |z|$ to complete the proof of the lemma. Changing variables, we have
$$
q_t(z) = C \int_{\R^2} e^{-2 |y|^2} \Phi_{\sigma_t}(\sqrt{t}y + z/2) \Phi_{\sigma_t}(\sqrt{t}y - z/2)  dy.
$$
Therefore, using that $ 0 \leq \Phi \leq 1$, 
\begin{align*}
| q_t(z) - q_t(0) | & \leq C \int_{\R^2} e^{-2 |y|^2} | \Phi_{\sigma_t}(\sqrt{t}y + z/2) - \Phi_{\sigma_t}(\sqrt{ty}) | dy + C \int_{\R^2} e^{-2 |y|^2} | \Phi_{\sigma_t}(\sqrt{t}y - z/2) - \Phi_{\sigma_t}(\sqrt{ty}) | dy \\
& \leq C |z| \int_{\R^2} e^{-2 |y|^2} \norme{ \con \Phi_{\sigma_t} }_{\R^2}  dy 
 \leq C \frac{|z|}{\sigma_t} \norme{ \con \Phi }_{\R^2} \int_{\R^2} e^{-2 |y|^2} dy.
\end{align*}
Since $\sigma_t = r_0 \sqrt{t} | \log t |^{\eps_0}$, we see that $\sup_{t \in [0,1]} \frac{\sqrt{t}}{\sigma_t} < \infty$, and the result follows.
\end{proof}

\subsection{Comparison between $\phi_{\delta}$ and $\psi_{\delta}$}

The following proposition justifies the introduction of the field $\psi_{\delta}$.
\begin{Prop}
\label{Prop:ComparisonFields}
There exist $C >0$ and $c>0$  such that for all $x > 0$, we have
\begin{equation}
\Pro \left( \sup_{n \geq 0} \norme{ \phi_{0,n} - \psi_{0,n} }_{[0,1]^2} \geq x \right) \leq C e^{-c x^2}.
\end{equation}
\end{Prop}
\begin{proof}
For $k \geq 1$, we introduce the quantity $D_k(x) := \phi_{k-1,k}(x) -  \psi_{k-1,k}(x)$.
The proof follows from an adaptation of Lemma 2.7 in \cite{DZZ18} as soon as we have the estimates
\begin{equation}
\label{eq:VarDiff1P}
\Var D_k(x) \leq C e^{-c k^{2\eps_0}}
\end{equation}
and
\begin{equation}
\label{eq:VarDiff2P}
\Var \left( \phi_{k}(x) - \phi_k(y) \right) +\Var \left( \psi_k(x) - \psi_k(y) \right) \leq 2^k | x - y|.
\end{equation}
(The estimate \eqref{eq:VarDiff1P} is weaker than that used in \cite[Lemma 2.7]{DZZ18} but still much stronger than required for the proof given there.)
Note that \eqref{eq:VarDiff2P} follows from Lemma  \ref{lem:FieldScale} and 
for \eqref{eq:VarDiff1P} we proceed as follows: first note that
$$
\E\left( \left( \phi_{k-1,k}(x) - \psi_{k-1,k}(x) \right)^2 \right)  = \int_{2^{-2k}}^{2^{-2k+2}} \int_{\R^2} p_{\frac{t}{2}}(y)^2 (1 - \Phi_{\sigma_t}(y))^2 dy dt.
$$
For every $y$, we have $p_{t/2}(y) (1-\Phi_{\sigma_t}(y)) \leq (2\pi t)^{-1} e^{-\sigma_t^2/t}$ since $0 \leq \Phi_{\sigma_t} \leq 1$ and $\Phi_{\sigma_t}(y) = 1$ for $|y| \leq \sigma_t$. Therefore,
$$
\E\left( \left( \phi_{k-1,k}(x) - \psi_{k-1,k}(x) \right)^2 \right)   \leq \int_{2^{-2k}}^{2^{-2k+2}} \frac{e^{-\frac{\sigma_t^2}{t}}}{2 \pi t} \int_{\R^2} p_{\frac{t}{2}}(y) dy dt \leq C e^{-c k^{2\eps_0}}.\qedhere
$$
\end{proof}
Let us point out that in fact $\sum_{n \geq 0} \E ( \norme{\phi_{n,n+1} - \psi_{n,n+1}}_{[0,1]^2}) < \infty$ holds but we won't use it. Since we will be working with two different approximations of the Gaussian free field, we introduce here some notation, referring to one field or the other. We will denote by $R_{a,b} := [0,a] \times [0,b]$ the rectangle of size $(a,b)$. We define
\begin{equation}
\label{DefX}
X_{a,b} := \sup_{n \geq 0} \norme{\phi_{0,n} - \psi_{0,n} }_{R_{a,b}}
\end{equation}
and  $X_a := X_{a,a} $ for the supremum norm of the difference between the two fields on various rectangles. 

\subsection{Length observables}
The symbol $L_{a,b}^{(n)}(\phi)$ (and similarly $L_{a,b}^{(n)}(\psi)$) will refer to the left-right distance of the rectangle $R_{a,b}$ for the length functional $e^{\xi \phi_{0,n}} ds$:
\begin{equation}
\label{DefLength}
L_{a,b}^{(n)}(\phi) := \inf_\pi \int_{\pi} e^{\xi \phi_{0,n}} ds,
\end{equation}
where $ds$ refers to the Euclidean length measure and the infimum is taken over all smooth curves $\pi$ connecting the left and right sides of $R_{a,b}$. We will sometimes consider a geodesic associated to this variational problem. Such a path exists by the Hopf-Rinow theorem and a compactness argument.

We introduce some notation for the quantiles associated to this observable: $\ell_{a,b}^{(n)}(\phi,p)$ (similarly $\ell_{a,b}^{(n)}(\psi,p)$) is such that $\Pro \left( L_{a,b}^{(n)}(\phi) \leq \ell_{a,b}^{(n)}(\phi)\right)  = p$. For high quantiles, we introduce $\bar{\ell}_{a,b}^{(n)}(\phi,p) := \ell_{a,b}^{(n)}(\phi,1-p)$. Note that $\ell_{a,b}^{(n)}(\phi,p)$ is increasing in $p$ whereas $\bar{\ell}_{a,b}^{(n)}(\phi,p)$ is decreasing in $p$. Note that both are well-defined, i.e., there are no Dirac deltas in the law of $L_{a,b}^{(n)}$. This follows from an application of the Cameron--Martin formula. We will also need  the notation  
\begin{equation}
\label{DefQuant}
\Lambda_{n}(\phi,p) := \max_{k \leq n} \frac{ {\bar{\ell}_k}(\phi,p)}{{\ell_k}(\phi,p)} \quad \text{ where } {\ell_k}(\phi,p) := \ell_{1,1}^{(k)}(\phi,p) \quad \text{ and } \quad {\bar{\ell}_k}(\phi,p) := \bar{\ell}_{1,1}^{(k)}(\phi,p).
\end{equation}

The following inequalities are straightforward:
\begin{equation}
e^{- \xi X_{a,b}} L_{a,b}^{(n)}(\psi)   \leq L_{a,b}^{(n)}(\phi) \leq e^{\xi X_{a,b}} L_{a,b}^{(n)}(\psi)  
\end{equation}
Therefore, using Proposition \ref{Prop:ComparisonFields} (and a union bound, if necessary), we obtain that for some $C > 0$ (depending only on $a$ and $b$), for any $\eps>0$ we have
\begin{align*}
e^{-\xi C \sqrt{ |\log \eps/C} |} \bar{\ell}_{a,b}^{(n)}(\psi,p+\eps) \leq \bar{\ell}_{a,b}^{(n)}(\phi,p) \leq e^{\xi C \sqrt{ |\log \eps/C} |} \bar{\ell}_{a,b}^{(n)}(\psi,p-\eps) \\
e^{-\xi C \sqrt{ |\log \eps/C} |} \ell_{a,b}^{(n)}(\psi,p-\eps) \leq \ell_{a,b}^{(n)}(\phi,p) \leq e^{\xi C \sqrt{ |\log \eps/C} |} \ell_{a,b}^{(n)}(\psi,p+\eps)
\end{align*}
In particular, there exists $C_p > 0$ such that, uniformly in $n$,
\begin{equation}
\label{eq:RatiosPsiPhi}
{\ell_n}(\psi,p/2) \geq \sqrt{C_p}^{-1} {\ell_n}(\phi,p), \quad {\bar{\ell}_n}(\psi,p/2) \leq \sqrt{C_p} {\bar{\ell}_n}(\phi,p) \quad \text{ and }  \quad \Lambda_n(\psi,p/2) \leq C_p \Lambda_n(\phi,p).
\end{equation} 

Now, we discuss how the scaling property of the field $\phi$ translates at the level of lengths. We will use  the following equality in law: for $a,b >0$ and $0 \leq m \leq n$,
\begin{equation}
\label{eq:scaling-length}
L_{a,b}^{(m,n)}(\phi) \overset{(d)}{=} 2^{-m} L_{a \cdot 2^m,b \cdot 2^m}^{(n-m)}(\phi).
\end{equation}%

Finally, for a rectangle $P$ with two marked opposite sides, we define
$L^{(n)}(P,\phi)$ to be the crossing distance between the two marked sides under the field $e^{\xi \phi_{0,n}}$. The marked sides will be clear from context: if we call $P$ a ``long rectangle,'' then we mean that the marked sides are the two shorter sides, so that $L^{(n)}(P,\phi)$ is the distance across $P$ ``the long way.''

\subsection{Outline of the proof and roles of $\phi_{\delta}$ and $\psi_{\delta}$}

The key idea of the proof is to obtain a self-bounding estimate associated to a measure of concentration of some observables, say rectangle crossing lengths. This is naturally expected because of the tree structure of our model. We introduce a general condition, which we call Condition (T), (see \eqref{eq:AssA}) which ensures a contraction in the  self-bounding estimate \eqref{eq:RecursiveIneq}, which relates a measure of concentration at scale $n$, the variance, with the measure of concentration that we inductively bound, $\Lambda_{n-K}$ (see \eqref{DefQuant}), which is at a smaller scale.

We then prove that this condition, which depends only on $\xi$ and on the field considered, is satisfied when $\xi \in (0,(2/d_2)^-)$. This proof uses a result taken from \cite{DG18} about the existence of an exponent for circle average Liouville first passage percolation and this is the reason we don't consider the simpler $\star$-scale invariant field with compactly-supported kernel but the field $\phi_{\delta}$, which can be compared to the circle average process by a result obtained in \cite{DGo18}. 

The roles of $\phi_{\delta}$ and $\psi_{\delta}$ in the proof are the following.
\begin{enumerate}
\item
Prove Russo-Seymour-Welsh estimates for $\phi$.

\item 
Prove tail estimates w.r.t low and high quantiles for both $\phi$ and $\psi$: 

\begin{enumerate}
\item
Lower tails: Use directly the RSW estimates together with a Fernique-type argument for the field $\psi$ with local independence properties.

\item
Upper tails: use a percolation/scaling argument, percolation using  $\psi$ and scaling using $\phi$.
\end{enumerate}

\item
Concentration of the log of the left-right distance: use Efron-Stein for the field $\psi$ (because of the local independence properties at each scale). This gives the same result for $\phi$.

\item
To conclude for the concentration of diameter and metric, this is essentially a chaining/scaling argument using only the field $\phi$. 
\end{enumerate}

\section{Russo-Seymour-Welsh estimates}

\label{Sec:AppCIGauss}

\subsection{Approximate conformal invariance}

In order to establish our RSW result, we first show an approximate conformal invariance property of the field. The arguments in this section are similar to those of \cite[Section 3.1]{DF18}. The difference is that the Gaussian kernel has infinite support.

Recall that $\phi_{\delta}(x) = \int_{\delta^2}^1 \int_{\R^2} p_{\frac{t}{2}}(x-y) W(dy,dt)$ where $p_t(x-y) = \frac{1}{2\pi t}e^{-\frac{|x-y|^2}{2t}}$. Consider a conformal map $F$ between two bounded, convex, simply-connected open sets $U$ and $V$ such that $\abs{F'} \geq 1$ on $U$, $\norme{F'}_U < \infty$ and $\norme{F''}_U < \infty$. (We point out here that the assumption $|F'| \geq 1$ will be obtained later on by starting from a very small domain; this is exactly the content of Lemma \ref{Lem:CrossLower}.) We consider another field $\tilde{\phi}_{\delta}(x) =  \int_{\delta^2}^1 \int_{\R^2} p_{\frac{t}{2}}(x-y) \tilde{W}(dy,dt)$  where $\tilde{W}$ is a white noise that we will couple with $W$ in order to compare $\phi_{\delta}$ and $ \tilde{\phi}_{\delta} \circ F$. The coupling goes as follows: for $y \in U$, $t \in (0,\infty)$, let $y' = F(y) \in V$ and $t' = t |F'(y)|^2 $ and set $\tilde{W}(dy',dt') = |F'(y)|^2 W(dy,dt)$. That is, for every $L^2$ function $\omega$ on $V \times (0, \infty)$,
$$
\int \omega(y',t') \tilde{W}(dy',dt') = \int \omega(F(y),t |F'(y)|^2) \abs{F'(y)}^{2} W(dy,dt)
$$
and both sides have variance $\norme{\omega}_{L^2}^2$. The rest of the white noises are chosen to be independent, i.e., $W_{|U^c \times (0, \infty)}$, $W_{|U \times (0, \infty)}$ and $\tilde{W}_{|V^c \times (0, \infty)}$ are jointly independent.

\begin{Lem}
Under this coupling, we can compare the two fields $\tilde{\phi} _{\delta} (F(x))$ and  $\phi_{\delta} (x)$ on a compact, convex subset $K$ of $U$ as follows,
\begin{equation}
\label{eq:Decompo}
\tilde{\phi} _{\delta} (F(x)) - \phi_{\delta} (x)  =  \phi_{\mathrm{L}}^{(\delta)}(x) +  \phi_{\mathrm{H}}^{(\delta)}(x),
\end{equation}
where $\phi_{\mathrm{L}}^{(\delta)}$ (L for low frequency noise)  is a smooth Gaussian field  whose $L^{\infty}$-norm on $K$ has uniform Gaussian tails, and $\phi_{\mathrm{H}}^{(\delta)}$ (H for high frequency noise) is a smooth Gaussian field with uniformly bounded pointwise variance (in $\delta$ and $x \in K$). Furthermore, $\phi_{\mathrm{H}}^{(\delta)}$ is independent of $(\phi_{\delta}, \phi_{\mathrm{L}}^{(\delta)})$. 
\end{Lem}
This aforementioned independence property will be crucial for our argument.

\begin{proof}

\textit{Step 1:} Decomposition. For fixed $F$ and small $\delta$, we decompose $\phi_{\delta} (x) - \tilde{\phi} _{\delta} (F(x)) = \phi_1^{(\delta)}(x) +  \phi_2^{(\delta)}(x) +\phi_3^{(\delta)}(x)$, where
\begin{align*}
\phi_1^{(\delta)}(x) = & \int_{U} \int_{\delta^2}^{\abs{F'(y)}^{-2}} \left( p_{\frac{t}{2}}(x-y) - p_{{\frac{t}{2}} \abs{F'(y)}^2}  \left( F(x) - F(y) \right) \abs{F'(y)}^2 \right)  W(dy,dt)   \\
= &\int_{U} \int_{\delta^2}^{\abs{F'(y)}^{-2}} \left( p_{\frac{t}{2}}(x-y) - p_{{\frac{t}{2}}}  \left( \frac{ F(x) - F(y)}{F'(y)} \right) \right)  W(dy,dt)\\
 \phi_2^{(\delta)}(x) = & \int_{U^c} \int_{\delta^2}^1 p_{\frac{t}{2}}(x-y) W(dy,dt) - \int_{V^c} \int_{\delta^2}^1 p_{\frac{t}{2}} \left( F(x)-y \right) \tilde{W}(dy,dt) \\
&  + \int_{U} \int_{\abs{F'(y)}^{-2}}^1 p_{\frac{t}{2}}(x-y)  W(dy,dt)    \\
\phi_3^{(\delta)}(x) = & - \int_{U} \int_{\delta^2 \abs{F'(y)}^{-2}}^{\delta^2} p_{\frac{t}{2}} \left( \frac{F(x)-F(y)}{ F'(y)} \right) W(dy,dt)
\end{align*}

Remark also that $ \phi_3^{(\delta)}$ is independent of $\phi_{\delta}$, $ \phi_1^{(\delta)}$, and $ \phi_2^{(\delta)}$. 

\textit{Step 2:} Conclusion, assuming uniform estimates. We will estimate $\phi_i^{(\delta)}$, $i=1,2,3$, over $K$. In what follows, we take $x,x' \in K$. We assume first the following uniform estimates:
$$
\E ( ( \phi_1^{(\delta)}(x) - \phi_1^{(\delta)}(x') )^2) \leq C \abs{x-x'}, \quad
 \E ( ( \phi_2^{(\delta)}(x) -  \phi_2^{(\delta)}(x') )^2 )  \leq C \abs{x-x'}, \quad  \E \left(\phi_{3}^{(\delta)} (x) \right) \leq C.
$$
An application of Kolmogorov's continuity criterion and Fernique's theorem give uniform Gaussian tails for $\phi_1^{(\delta)}$ and $ \phi_2^{(\delta)}$. We then set $ \phi_{\mathrm{H}}^{(\delta)} :=  \phi_3^{(\delta)} $ and  $ \phi_{\mathrm{L}}^{(\delta)} :=  \phi_1^{(\delta)} +  \phi_2^{(\delta)}$.

\textit{Step 3:} Uniform estimates.

\textbf{First term.} We prove that $\E ( ( \phi_1^{(\delta)}(x) -  \phi_1^{(\delta)}(x') )^2) \leq C \abs{x-x'}$ by controlling
$$
\int_0^1 \int_U \left( p_{\frac{t}{2}}(x-y) - p_{\frac{t}{2}} \left(\frac{F(x) -F(u)}{F'(y)} \right) -  p_{\frac{t}{2}}(x'-y) +  p_{\frac{t}{2}} \left(\frac{F(x') -F(u)}{F'(y)}\right)\right)^2 dy dt 
$$
By introducing $p(x) = e^{-\frac{|x|^2}{2}}$ and by a change of variable $ t \leftrightarrow 2 t^2$, it is equivalent (up to a multiplicative constant) to bound from above the quantity
\begin{equation}
\label{eq:Bound}
\int_0^1 \frac{dt}{t^3} \int_U \left( p\left(\frac{x-y}{t}\right) - p\left(\frac{F(x) -F(y)}{t F'(y)} \right) -  p\left(\frac{x'-y}{t}\right) +  p \left(\frac{F(x') -F(y)}{tF'(y)}\right)\right)^2 dy .
\end{equation}
We will estimate this term by considering the case where $t \leq \sqrt{|x-x'|}$ and the case where $t \geq \sqrt{|x-x'|}$.

\textit{Step 3.(A):} Case $t \geq \sqrt{|x-x'|}$. Using the identity $| x- y |^2 + |x'-y|^2 = \frac{1}{2} |x-x'|^2 + 2 | y - \frac{x+x'}{2} |^2$ and the inequality $1-e^{-z} \leq z$, we get
\begin{equation}
\label{eq:BoundBis}
\int_{U} \left( p\left(\frac{x-y}{t}\right) - p\left( \frac{x'-y}{t} \right) \right)^2 dy \leq C t^2 (1-e^{-\frac{|x-x'|^2}{4t^2}})  \leq C \abs{x-x'}^2.
\end{equation}
Similarly,
\begin{equation}
\label{eq:BoundBis2}
\int_{U} \left( p\left(\frac{F(x)-F(y)}{tF'(y)}\right) - p\left( \frac{F(x')-F(y)}{tF'(y)} \right) \right)^2 dy  \leq C \abs{F(x)-F(x')}^2\leq C\abs{x-x'}^2,
\end{equation}
where the constant $C$ depends on $\|F'\|_U$. Then the corresponding part in \eqref{eq:Bound} is bounded from above by $|x-x'|^2 \int_{\sqrt{|x-x'|}}^1 \frac{dt}{t^3} \leq C |x-x'|$.  

\textit{Step 3.(B):} For $t \leq \sqrt{|x-x'|}$, using the Taylor inequality $| F(x) - F(y) - F'(y)(x-y) | \leq \frac{1}{2} \norme{F''}_{U} |x-y|^2$  and the mean value inequality (as we have assumed that $K$ is convex),
\begin{equation}
\label{eq:Taylor}
\abs{p\left(\frac{x-y}{t}\right) - p\left(\frac{F(x) -F(y)}{t F'(y)} \right) } 
\leq C \frac{|x-y|^2}{t} \left( \frac{|x-y|}{t} + \frac{|x-y|^2}{t} \right) e^{-\frac{1}{2t^2} \inf_{\alpha \in(0,1)} \abs{\alpha(x-y) + (1-\alpha) \frac{F(x)-F(y)}{F'(y)}}^2}.
\end{equation}

\textit{Step 3.(B): case (a).} If $y\in B(x,\eps)$ for $\eps$ small enough (depending only on $\norme{F''}_U$), we have, using again $| F(x) - F(y) - F'(y)(x-y) | \leq \frac{1}{2} \norme{F''}_{U} |x-y|^2$, uniformly in $\alpha \in (0,1)$,
$$
\abs{\alpha(x-y) + (1-\alpha) \frac{F(x)-F(y)}{F'(y)}} \geq |x-y| - \frac{1}{2} \norme{F''}_{U} |x-y|^2 \geq \frac{1}{2} |x-y|.
$$
Therefore, for such $y$'s we have, coming back to \eqref{eq:Taylor},
$$
\abs{p\left(\frac{x-y}{t}\right) - p\left(\frac{F(x) -F(y)}{t F'(y)} \right) }\leq C \frac{|x-y|^3}{t^2} e^{-\frac{|x-y|^2}{4t^2}}.
$$
For this case we get the bound
$$
\int_{B(x,\eps)} \left( p\left(\frac{x-y}{t}\right) - p\left(\frac{F(x) -F(y)}{t F'(y)} \right)  \right)^2 dy \leq C \int_{B(x,\eps)} \frac{|x-y|^6}{t^4} e^{-\frac{|x-y|^2}{2t^2}}  =  C t^{-2} \E (\abs{B_{t^2}}^6) \leq C t^4.
$$
where $B_t$ denotes a two-dimensional Gaussian variable with covariance matrix $t$ times the identity. This term contributes to \eqref{eq:Bound} as $ C\int_{0}^{\sqrt{|x-x'|}} \frac{dt}{t^3} t^4 \leq C |x-x'|$. 

\textit{Step 3.(B): case (b).} Now, for $t \leq \sqrt{|x-x'|}$ and $y \in U \setminus B(x,\eps)$ we write 
$$
\int_0^{\sqrt{|x-x'|}}\frac{dt}{t^3} \int_{U \setminus B(x,\eps)} p\left(\frac{x-y}{t}\right)^2 dy  \leq C \int_0^{\sqrt{|x-x'|}}\frac{dt}{t} \Pro ( \abs{B_{t^2}} > \eps ) \leq  C \int_0^{\sqrt{|x-x'|}}\frac{dt}{t} e^{-\frac{\eps^2}{2 t^2}} \leq C \abs{x-x'},
$$
and similarly
$$
\int_0^{\sqrt{|x-x'|}}\frac{dt}{t^3} \int_{U \setminus B(x,\eps)} p\left(\frac{F(x)-F(y)}{tF'(y)}\right)^2 dy  \leq  C \abs{x-x'},
$$
where the constant $C$ depends on $\|F'\|_U$ and $\|(F^{-1})'\|_U$.

Applying Step~3.(A) and then Step~3.(B) twice (once for $x$ and then again for $x'$) to \eqref{eq:Bound}, we get $\E ( ( \phi_1^{(\delta)}(x) - \phi_1^{(\delta)}(x') )^2) \leq C \abs{x-x'}$.

\textbf{Second term.}  We want to prove here that  $\E ( ( \phi_2^{(\delta)}(x) -  \phi_2^{(\delta)}(x') )^2 )  \leq C \abs{x-x'}$. Note  that three terms contribute to $\delta \phi_2$. The third one is a nice Gaussian field independent of $\delta$. The first two terms are similar, so we will just focus on the first one, namely $ \phi_{2,1}^{(\delta)}(x) :=  \int_{U^c} \int_{\delta^2}^1 p_{\frac{t}{2}}(x-y) W(dy,dt)$. We have, similarly to \eqref{eq:Bound} and \eqref{eq:BoundBis},
\begin{align*} 
\E \left( \left(  \phi_{2,1}^{(\delta)}(x) - \phi_{2,1}^{(\delta)}(x') \right)^2 \right) & = \int_{\delta^2}^1 \int_{U^c} \left( p_{\frac{t}{2}}(x-y) -  p_{\frac{t}{2}}(x'-y)  \right)^2 dy dt \\
& \leq C \int_{0}^1 \frac{dt}{t^3} \int_{U^c} \left( p\left( \frac{x-y}{t} \right) - p\left( \frac{x'-y}{t}\right)  \right)^2 dy \\
& \leq C \int_0^{\sqrt{|x-x'|}} \frac{dt}{t^3} \int_{U^c} p\left( \frac{x-y}{t} \right) + p\left( \frac{x'-y}{t}\right) dy  + C | x-x'|.
\end{align*}
The remaining term  can be controlled as follows (noting the symmetry between $x$ and $x'$):
$$
\int_{0}^{\sqrt{|x-x'|}} \frac{dt}{t} \int_{U^c}  \frac{1}{t^2} e^{- \frac{|x-y|^2}{2 t^2}} dy \leq C \int_{0}^{\sqrt{|x-x'|}}  \frac{dt}{t} \Pro(\abs{ B_{t^2} } > d) \leq  C \int_{0}^{\sqrt{|x-x'|}}  \frac{dt}{t} e^{-\frac{d^2}{2t^2}} \leq C  \abs{ x-x'}.
$$
where  $d = d(K,U^c)$. Thus $\E ( ( \phi_2^{(\delta)}(x) -  \phi_2^{(\delta)}(x') )^2 )  \leq C \abs{x-x'}$.

\textbf{Third term.} We give here a bound on the pointwise variance of $\phi_{3}^{(\delta)}$. By using $\abs{\frac{F(x)-F(y)}{F'(y)}} \geq \frac{\abs{x-y}}{C}$ we get $\E (  \phi_3^{(\delta)}(x)^2 ) \leq \int_{c \delta^2}^{\delta^2} \frac{dt}{t} \int_{\R^2} \frac{ e^{-\frac{|x-y|^2}{Ct}}}{t} dy \leq C$.
\end{proof}

\subsection{Russo-Seymour-Welsh estimates}

\label{Sec:RSWreminder}

The main result of this section is the following RSW estimate. It shows that appropriately-chosen quantiles of crossing distances of ``long'' and ``short'' rectangles at the same scale can be related by a multiplicative factor that is uniform in the scale. This is the equivalent of Theorem 3.1 from \cite{DF18} but with the field mollified by the heat kernel instead of a compactly-supported kernel. It holds for any fixed $\xi > 0$.
\begin{Prop}[RSW estimates for $\phi_{\delta}$] If $[A,B] \subset (0,\infty)$, there exists $C > 0$ such that for $(a,b)$, $(a',b') \in [A,B]$ with $\frac{a}{b} < 1 < \frac{a'}{b'}$, for $n \geq 0$ and $\eps < 1/2$, we have,
\label{Prop:RSWphi}
\begin{equation}
\label{eq:LowQuantilesPhi}
\ell_{a',b'}^{(n)}(\phi,\eps/C) \leq C \ell_{a,b}^{(n)}(\phi,\eps)  e^{C \sqrt{\log | \eps/C |}};
\end{equation}
\begin{equation}
\label{eq:HighQuantilesPhi}
\bar{\ell}_{a',b'}^{(n)}(\phi,3\eps^C) \leq C \bar{\ell}_{a,b}^{(n)}(\phi,\eps) e^{C \sqrt{\log | \eps/C |}}.
\end{equation}
\end{Prop}

\noindent The following corollary then follows from Propositions~\ref{Prop:ComparisonFields} and~\ref{Prop:RSWphi}.
\begin{Cor}[RSW estimates for $\psi_{\delta}$] Under the same assumptions as used in Proposition~\ref{Prop:RSWphi}, we have
\label{Cor:RSWpsi}
\begin{equation}
\label{eq:LowQuantilesPsi}
\ell_{a',b'}^{(n)}(\psi,\eps/C) \leq C \ell_{a,b}^{(n)}(\psi,\eps)  e^{C \sqrt{\log | \eps/C |}}
\end{equation}
and
\begin{equation}
\label{eq:HighQuantilesPsi}
\bar{\ell}_{a',b'}^{(n)}(\psi,3\eps^C) \leq C \bar{\ell}_{a,b}^{(n)}(\psi,\eps) e^{C \sqrt{\log | \eps/C |}}.
\end{equation}
\end{Cor}
\noindent We point out that the constants $C$ in \eqref{eq:LowQuantilesPsi} and \eqref{eq:HighQuantilesPsi} are not equal to those in \eqref{eq:LowQuantilesPhi} and \eqref{eq:HighQuantilesPhi}. The remaining parts of the section will only deal with approximations associated with $\phi$ so we will omit this dependence in the various observables.

We describe below the main lines of the argument. Consider $R_{a,b}$ and $R_{a',b'}$, two rectangles with respective side lengths $(a,b)$ and $(a',b')$ satisfying $\frac{a}{b} < 1 < \frac{a'}{b'}$. Suppose that we could take a conformal map $F : R_{a,b} \to R_{a',b'}$ mapping the long left and right sides of $R_{a,b}$ to the short left and right sides of $R_{a',b'}$. (This is not in fact possible since there are only three degrees of freedom in the choice of a conformal map, but for the sake of illustration we will consider this idealized setting first.) Then the proof goes as follows. 

Take a geodesic $\tilde{\pi}$ for $\tilde{\phi}_{0,n}$ for the left-right crossing of $R_{a,b}$. Then, using the coupling \eqref{eq:Decompo}, we have
\begin{align*}
L^{\phi_{0,n}}(R_{a',b'})  \leq L^{\phi_{0,n}}(F(\tilde{\pi}))   = \int_{0}^{T} e^{\xi \phi_{0,n} (F(\tilde{\pi}(t)))} | F'(\tilde{\pi} (t)) | \cdot | \tilde{\pi}'(t) | dt  &\leq \norme{ F' }_{R_{a,b}} \int_{\tilde{\pi}} e^{\xi  (\tilde{\phi}_{0,n}  + \delta \phi_{\mathrm{L}} + \delta \phi_{\mathrm{H}})} ds \\
& \leq \norme{ F' }_{R_{a,b}} e^{\xi \norme{\delta \phi_{\mathrm{L}} }_{R_{a,b}}} \int_{\tilde{\pi}}  e^{\xi  \tilde{\phi}_{0,n}  } e^{\xi \delta \phi_{\mathrm{H}}} ds .
\end{align*}
It is essential that $\tilde{\pi}$ is $\tilde{\phi}_{0,n}$ measurable and $\tilde{\phi}_{0,n}$ is independent of $\delta \phi_{\mathrm{H}}$. Then, we can use the following lemma. 

\begin{Lem}
\label{Inequality} If $\Gamma$ is a continuous field and $\Psi$ is an independent continuous centered Gaussian field with pointwise variance bounded above by $\sigma^2 >0$, then we have, as long as $\eps$ is sufficiently small compared to $\sigma^2$,
\begin{enumerate}
\item
$\ell_{1,1} (\Gamma + \Psi,\eps) \leq e^{\sqrt{2 \sigma^2  \log \eps^{-1} }} \ell_{1,1}(\Gamma,2\eps) $;

\item 
$\bar{\ell}_{1,1}(\Gamma + \Psi,2\eps) \leq e^{\sqrt{2 \sigma^2  \log \eps^{-1} }} \bar{\ell}_{1,1}(\Gamma,\eps)$.
\end{enumerate} 
\end{Lem}
\begin{proof}
Fix $s := \sqrt{2 \sigma^2 \log \eps^{-1}}$ throughout the proof. Let $\pi(\Gamma)$ be a geodesic associated with the left--right crossing length for the field $\Gamma$, and define the measure $\mu$ on $\pi(\Gamma)$ by $\mu(ds) = L_{1,1}(\Gamma)^{-1}e^{\Gamma}ds$, so $\int_{\pi(\Gamma)} e^{\Gamma}ds = 1$. 
Conditionally on $\Gamma$, 
using Jensen's inequality with $\alpha = \frac{s}{2 \sigma^2} = \sqrt{(\log \eps^{-1})/(2\sigma^2)}$, which is greater than $1$ for small enough $\eps$, and Chebyshev's inequality, we have
\begin{equation}
\label{eq:MomentEst}
\Pro \left(  \int_{\pi(\Gamma)} e^{\Gamma + \Psi} ds > e^s L_{1,1}(\Gamma) ~ | ~ \Gamma \right)  \leq \Pro \left( \int_{\pi(\Gamma)} e^{\alpha \Psi}d\mu \geq e^{\alpha s} ~ | ~ \Gamma \right) \leq e^{\frac{1}{2} \alpha^2 \sigma^2} e^{-\alpha s} = e^{-\frac{s^2}{2\sigma^2}} = \eps.
\end{equation}
To bound from above $L_{1,1}(\Gamma + \Psi)$, we take a geodesic for $\Gamma$ and use the moment estimate \eqref{eq:MomentEst}.  We start with the left tail. Still with $s := \sqrt{2 \sigma^2 \log \eps^{-1}}$, we have 
\begin{align*}
\Pro \left( L_{1,1}(\Gamma) \leq \ell_{1,1}(\Gamma + \Psi,\eps) e^{-s}  \right) & \leq \Pro \left( L_{1,1}(\Gamma + \Psi) \leq e^s L_{1,1}(\Gamma), L_{1,1}(\Gamma) \leq \ell_{1,1}(\Gamma + \Psi,\eps) e^{-s} \right) \\
& \quad + \Pro \left( L_{1,1}(\Gamma + \Psi) > e^s L_{1,1}(\Gamma) \right) \\
& \leq \Pro \left( L_{1,1}(\Gamma + \Psi) \leq \ell_{1,1}(\Gamma + \Psi,\eps)  \right)  +  \Pro \left(  \int_{\pi(\Gamma)} e^{\Gamma + \Psi} ds > e^s L_{1,1}(\Gamma)\right) \leq 2 \eps.
\end{align*}
For the right tail, we have similarly that
\begin{align*}
\Pro &\left( L_{1,1}(\Gamma + \Psi) \geq \bar{\ell}_{1,1}(\Gamma, \eps) e^s  \right)  \\& \leq \Pro \left( L_{1,1}(\Gamma + \Psi) \geq \bar{\ell}_{1,1}(\Gamma, \eps) e^s,   \bar{\ell}_{1,1}(\Gamma,\eps) \geq L_{1,1}(\Gamma) \right) +  \Pro \left( L_{1,1}(\Gamma) \geq  \bar{\ell}_{1,1}(\Gamma,\eps) \right)  \\
& \leq \Pro \left(L_{1,1}(\Gamma + \Psi) \geq e^s L_{1,1}(\Gamma)\right)   +  \eps  \leq 2 \eps,
\end{align*}
which concludes the proof of the lemma.
\end{proof}

The previous reasoning does not apply directly to rectangle crossing lengths but provides the following proposition. Recall that $K$ is a compact subset of $U$. Let $A,B$ be two boundary arcs of $K$ and denote by $L$ the distance from $A$ to $B$ in $K$ for the metric $e^{\xi\phi_{0,n}}ds$; we denote $A' :=F(A)$, $B' :=F(B)$, $K' :=F(K)$, and $L'$ is the distance from $A'$ to $B'$ in $K'$ for $e^{\xi \tilde{\phi}_{0,n}}ds$. Recall that we have $| F'| \geq 1$ on $U$. In the application we will achieve this by scaling $U$ to be sufficiently small. 
\begin{Prop} 
\label{Prop:RSWconf}We have the following comparisons between quantiles. There exists $C > 0$ such that
\begin{enumerate}
\item
if for some  $l > 0$ and  $\eps < 1/2$, $\mathbb{P}\left(  L \leq l \right) \geq \eps$, then $\mathbb{P} \left( L' \leq l' \right) \geq \eps/4$ with $l' = \norme{F'}_K e^{C  \sqrt{\abs{\log\eps / 2C}}} $.

\item
if for some  $l > 0$ and  $\eps < 1/2$, $\mathbb{P}\left( L \leq l \right) \geq 1- \eps$, then $\mathbb{P} \left( L' \leq l' \right) \geq 1 - 3 \eps $
with $l' = \norme{F'}_K e^{C  \sqrt{\abs{\log\eps / 2C}}}$.
\end{enumerate}
\end{Prop}

Now, we want to prove a similar result for rectangle crossing lengths. We will need the three following lemmas that were used in \cite{DF18}. The first one is a geometrical construction, the second one is a complex analysis result and the last one comes essentially from \cite{Pitt82} together with an approximation argument. In these lemmas, by ``crossings'' we mean continuous path from marked sides to marked sides.

\begin{Lem} [Lemma~4.8 of \cite{DF18}]
\label{Lem:CrossLower}
If $a$ and $b$ are two positive real numbers with $a < b$, there exists $j = j(b/a)$ and $j$ rectangles isometric to $[0, a/2] \times [0, b/2]$  such that if $\pi$ is a left-right crossing of the rectangle $ [0,a] \times [0,b]$, at least one of the $j$ rectangles is crossed in the thin direction by a subpath of that crossing. 
\end{Lem}

\begin{Lem}[Step 1 in the proof of Theorem 3.1 in \cite{DF18}]
\label{Lem:ComparisonEllipses}
If $a/b < 1$ and $a'/b' > 1$, there exists $m,p \geq 1$ and two ellipses $E_p,E'$ with marked arcs $(AB)$, $(CD)$ for $E_p$ and $(A'B')$, $(C'D')$ for $E'$ such that:
\begin{enumerate}
\item
Any left-right crossing of $[0, a/2^p] \times [0,b/2^p]$ is a crossing of $E_p$. 

\item
Any crossing of $E'$ is a left-right crossing of $[0,a'] \times [0,b']$. 

\item
When dividing the marked sides of $E_p$ into $m$ subarcs of equal length, for any pair of such subarcs (one on each side), there exists a conformal map $F : E_p \to E'$ and the pair of subarcs is mapped to subarcs of the marked sides of $E'$.

\item 
For each pair, the associated map $F$ extends to a conformal equivalence $U \to V$ where $\overline{E_p} \subset U$, $\overline{E'} \subset V$ and $|F'| \geq 1$ on $U$. 
\end{enumerate} 
\end{Lem}

We refer the reader to Figure \ref{fig:illustration-lemma} for an illustration.

\begin{figure}[ht]
\centering
\includegraphics[scale=1]{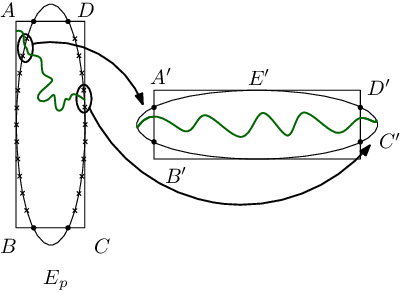}
\caption{Illustration of Lemma \ref{Lem:ComparisonEllipses}.} 
\label{fig:illustration-lemma}
\end{figure}

\begin{Lem}[Positive association and square-root-trick]
\label{Lem:FKG} If $k \geq 2$ and $(R_1, \dots, R_k)$ denote a collection of $k$ rectangles, then, for $(x_1, \dots, x_k) \in (0,\infty)^k$, we have
$$
\Pro \left( L^{(n)}(R_1) > x_1, \dots, L^{(n)}(R_k) > x_k \right) \geq \Pro \left( L^{(n)}(R_1) > x_1 \right) \cdots   \Pro \left( L^{(n)}(R_k) > x_k \right).
$$
An easy consequence of this positive association is the so-called ``square-root-trick":
$$
\max_{i \leq k} \Pro \left( L^{(n)}(R_i) \leq x_i \right) \geq 1 - \left( 1 - \Pro \left( \exists i \leq k : L^{(n)}(R_i) \leq x_i \right) \right)^{1/k}.
$$
\end{Lem}

The main result of this section, Proposition \ref{Prop:RSWphi}, is a rephrasing of the following one.
\begin{Prop}
We have the following comparisons between quantiles. If $a/b < 1$ and $a'/b' > 1$, there exists $C > 0$ such that, for any $\eps \in (0,1/2)$,
\begin{enumerate}
\item
\label{item:CompLowQuantilies}
if $\mathbb{P} \left( L_{a,b}^{(n)} \leq l \right) \geq \eps$, then $\mathbb{P}\left( L_{a',b'}^{(n)} \leq C l e^{C \sqrt{\abs{\log \eps / C}}} \right) \geq \eps/C$,

\item
\label{item:CompHighQuantiles}
and if  $\mathbb{P} \left( L_{a,b}^{(n)} \leq l \right) \geq 1 - \eps$, then $\mathbb{P}\left( L_{a',b'}^{(n)} \leq C l e^{C \sqrt{\abs{\log \eps /C}}} \right) \geq 1 - 3 \eps^{1/C}$.
\end{enumerate}
\end{Prop}

\begin{proof}
We provide first a comparison between low quantiles and then a comparison between high quantiles.

\textit{Step 1:}  Comparison of small quantiles. Suppose $\mathbb{P} (L_{a,b}^{(n)} \leq l ) \geq \eps$. By Lemma \ref{Lem:CrossLower} and union bound, $\mathbb{P} (L_{a/2,b/2}^{(n)}  \leq l ) \geq \eps/j$. Furthermore, by iterating, we have $\mathbb{P}( L_{a/2^p,b/2^p}^{(n)} \leq l ) \geq \eps / j^p$. Under this event, by Lemma \ref{Lem:ComparisonEllipses}, there exists a crossing of $E_p$ between two subarcs of $E_p$ (one on each side) hence with probability at least $ \eps/ (j^{p} m^2 )$, one of these crossings has length at most $l$. By the left tail estimate Proposition \ref{Prop:RSWconf} and Lemma \ref{Lem:ComparisonEllipses}, we obtain a $C > 0$ (depending also on $\norme{F'}_{\overline{E_p}}$) such that for all $\eps, l > 0$:
$$
\mathbb{P} \left( L_{a,b}^{(n)} \leq l \right) \geq \eps \Rightarrow \mathbb{P}\left( L_{a',b'}^{(n)} \leq C l e^{C \sqrt{\abs{\log \eps / (2C j^p m^2)}}} \right) \geq \eps/(4 j^p m^2),
$$
hence the first assertion.

\textit{Step 2:} Comparison of high quantiles.   Now suppose $\mathbb{P} (L_{a,b}^{(n)} \leq l ) \geq 1- \eps$. By Lemma \ref{Lem:CrossLower} (to start with a crossing at a lower scale) and Lemma \ref{Lem:FKG} (square-root-trick), we have $\mathbb{P} (L_{a/2,b/2}^{(n)}  \leq l ) \geq 1-  \eps^{1/j}$. Furthermore, by iterating, we have $\mathbb{P}( L_{a/2^p,b/2^p}^{(n)} \leq l ) \geq 1 - \eps^{1/j^p}$. On the event $\lbrace L_{a/2^p,b/2^p}^{(n)} \leq l \rbrace$, the ellipse $E_p$ from Lemma \ref{Lem:ComparisonEllipses} has a crossing of length $\leq l$ between two marked arcs. Again by subdividing each its marked arcs into $m$ subarcs and applying the square-root trick, we see that for at least one pair of subarcs, there is a crossing of length $\leq l$ with probability $\geq 1 - \eps^{j^{-p} m^{-2}}$. Combining with the right-tail estimate Proposition \ref{Prop:RSWconf} and Lemma \ref{Lem:ComparisonEllipses}, we get:
\begin{equation}
\mathbb{P} \left( L_{a,b}^{(n)} \leq l \right) \geq 1 - \eps \Rightarrow \mathbb{P}\left( L_{a',b'}^{(n)} \leq C l e^{C \sqrt{\abs{\log \eps /C}}} \right) \geq 1 - 3 \eps^{1/C},
\end{equation}
which completes the proof.
\end{proof} 

\begin{Rem}
The importance of the Russo-Seymour-Welsh estimates comes from the following: percolation arguments/estimates work well when taking small quantiles associated with short crossings and high quantiles associated with long crossings. Thanks to the RSW estimates, we can instead keep track only of low and high quantiles associated to the unit square crossing, ${\ell_n}(p)$ and ${\bar{\ell}_n}(p)$.
\end{Rem}

\section{Tail estimates with respect to fixed quantiles}

\label{Sec:TailEstimates}

\paragraph{Lower tails.} This is where we take $r_0$ small enough (recall the definition \eqref{Def:sigma}) to obtain some small range of dependence of the field $\psi$ so that a Fernique-type argument works.

\begin{Prop}[Lower tail estimates for $\psi$]
We have the following lower tail estimate: for $p$ small enough, but fixed, there is a constant $C$ so that for all $s > 0$,
\begin{equation}
\label{eq:LowerTailsPsi}
\Pro \left( L_{1,3}^{(n)}(\psi) \leq e^{-s} {\ell_n}(\psi,p) \right) \leq C e^{- c s^2}.
\end{equation}
\end{Prop}

\begin{proof}
The RSW estimate \eqref{eq:LowQuantilesPsi} gives
\begin{equation}
\label{eq:RSWreUse}
\mathbb{P}\left( L_{3,3}^{(n)}(\psi) \leq l \right) \leq \eps \Rightarrow \mathbb{P}\left( L_{1,3}^{(n)}(\psi) \leq l C^{-1} e^{-C \xi \sqrt{ | \log C \eps |}} \right) \leq C \eps
\end{equation}
Now, if $L_{3,3}^{(n)}(\psi)$ is less than $l$, then both $[0,1] \times [0,3]$ and $[2,3] \times [0,3]$ have a left-right crossing of length $\leq l$ and the restrictions of the field to these two rectangles are independent (if $r_0$ defined in \eqref{Def:sigma} is small enough). Consequently, 
\begin{equation}
\label{decaysquare}
\mathbb{P}\left( L_{3,3}^{(n)}(\psi) \leq l \right) \leq \mathbb{P}\left( L_{1,3}^{(n)}(\psi) \leq l \right)^2
\end{equation}
Take $p_0$  small, such that $C^2 p_0 < 1$ where $C$ is the constant in \eqref{eq:RSWreUse} and set $r_0^{(n)} := \ell_{3,3}^{(n)}(\psi,p_0)$. (This is not related to $r_0$, defined previously.) For $i \geq 0$, set
\begin{align}
\label{eq:receps}
p_{i+1} & := (C p_i)^2 \\
\label{eq:recr}
r_{i+1}^{(n)} & := r_{i}^{(n)} C^{-1} \exp (-C \xi \sqrt{ | \log (C p_i) |})
\end{align}
By induction we get, for $i \geq 0$,
\begin{equation}
\label{eq:LowerEst}
\mathbb{P} ( L_{3,3}^{(n)}(\psi) \leq r_i^{(n)} ) \leq p_i
\end{equation}
Indeed, the case $i = 0$ follows by definition and then notice that the RSW estimate \eqref{eq:RSWreUse} under the induction hypothesis implies that $\mathbb{P} ( L_{3,3}^{(n)}(\psi) \leq r_i^{(n)}) \leq p_i \Rightarrow \mathbb{P} ( L_{1,3}^{(n)}(\psi) \leq r_{i+1}^{(n)}) \leq C p_i$ which gives, using \eqref{decaysquare}, $\mathbb{P} ( L_{3,3}^{(n)}(\psi) \leq r_{i+1}^{(n)}) \leq \mathbb{P} ( L_{1,3}^{(n)}(\psi) \leq r_{i+1}^{(n)})^2 \leq  (C p_i )^2 = p_{i+1}$.

From \eqref{eq:receps} we get $p_i = (p_0 C^2)^{2^i} C^{-2}$ and from \eqref{eq:recr} we have the lower bound, for $i \geq 1$,
$$
r_i^{(n)} \geq \ell_{3,3}^{(n)}(\psi,p_0) C^{-i} e^{-C \xi \sum_{k = 0}^{i-1} \sqrt{| \log (C p_k)  |} } \geq  \ell_{3,3}^{(n)}(\psi,p_0) e^{-C i} e^{-C \xi \sqrt{|\log p_0 C^2 |} 2^{i/2}}.
$$
Our estimate \eqref{eq:LowerEst} then takes the form, for $i \geq 0$,
$$
\mathbb{P} \left( L_{3,3}^{(n)}(\psi) \leq \ell_{3,3}^{(n)}(\psi,p_0) e^{-C i} e^{-\xi C \sqrt{|\log p_0 C^2 |} 2^{i/2}} \right) \leq \left( p_0 C^2 \right)^{2^i} C^{-2}.
$$
This can be rewritten, taking $i = \lfloor 2 \log_2 s \rfloor$, as
$$
\mathbb{P}\left( L_{3,3}^{(n)}(\psi) \leq \ell_{3,3}^{(n)}(\psi,p_0) C^{-1} e^{-C \log s} e^{-\xi s} \right) \leq  e^{-c s^2}
$$
for $s > 2$ with absolute constants.
We obtain the statement of the proposition by using again the RSW estimates.\end{proof}

Using the comparison result between $\phi$ and $\psi$ (Proposition \ref{Prop:ComparisonFields}), we get the following corollary.
\begin{Cor}[Lower tail estimates for $\phi$]
For $p$ small enough, but fixed, for all $s > 0$ we have a constant $C<\infty$ so that
\begin{equation}
\label{eq:LowerTailsPhi}
\Pro \left( L_{1,3}^{(n)}(\phi) \leq e^{-s} {\ell_n}(\phi,p) \right) \leq C e^{- c s^2}.
\end{equation}
\end{Cor}

\paragraph{Upper tails.} The proof for the upper tails is similar to the one of Proposition 5.3 in \cite{DF18}. The main difference is that we have to switch between $\phi$ and $\psi$, so that we can use the independence properties of $\psi$ together with the scaling properties of $\phi$. Before stating the proposition, we refer the reader to \eqref{DefQuant} for the definition of $\Lambda_{n}(\phi,p)$. In constract with the lower tails estimates which are relative to ${\ell_n}(\phi,p)$, we do not know how to prove (at least a priori) the analogous result for the upper tails with ${\bar{\ell}_n}(\phi,p)$ only. However, we can prove it by replacing ${\bar{\ell}_n}(\phi,p)$ by $\Lambda_n(\phi,p) {\ell_n}(\phi,p)$ and this is the content of the following proposition.
\begin{Prop}[Upper tail estimates for $\phi$] For $p$ small enough, but fixed, we have a constant $C<\infty$ so that for all $n \geq 0$ and $s >2$,
\begin{equation}
\label{eq:UpperTailsPhi}
\Pro \left( L_{3,1}^{(n)}(\phi) \geq e^{s} \Lambda_{n}(\phi,p) {\ell_n} (\phi,p) \right) \leq C e^{c \frac{s^2}{\log s}}.
\end{equation}
\end{Prop}

\begin{proof} 

The proof uses percolation and scaling arguments. A percolation argument is used to build a crossing of a larger rectangle from smaller annular circuits, and then a scaling argument is used to relate quantiles of these annular crossings to crossing quantiles of the larger rectangle.

\begin{figure}[ht]
\centering
\includegraphics[scale=1]{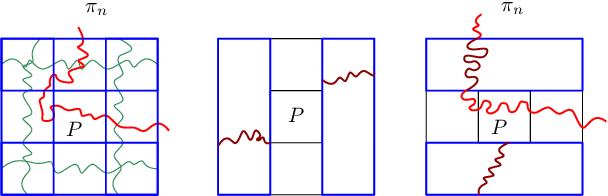}
\caption{Four blue rectangles are surrounding the square $P$. Left-right geodesics associated to the long and short rectangles surrounding $P$ are drawn in green and brown respectively. Any geodesic $\pi_n$, here in red, which intersects $P$ has to cross the green circuit and to induce a short crossing of one of the four rectangles.\label{fig:surroundingrectangles}}
\end{figure}

\textit{Step 1:} Percolation argument.  To each unit square $P$ of $\Z^2$, we associate the four crossings of long rectangles of size $(3,1)$  surrounding $P$, each comprising three squares on one side of the eight-square annulus surrounding $P$, as illustrated in Figure~\ref{fig:surroundingrectangles}. We define $S^{(n)}(P,\psi)$ to be the sum of the four crossing lengths, and declare the site $P$ to be open when the event $\{ S^{(n)}(\psi,P) \leq 4 \bar{\ell}_{3,1}^{(n)}(\psi,p) \}$ occurs. This occurs with probability at least $1 - \eps(p)$, where $\eps(p)$ goes to zero as $p$ goes to zero (recall that $\Pro(L_{3,1}^{(n)}(\psi) \leq \bar{\ell}_{3,1}^{(n)}(p)) = 1 - p$). Using a highly supercritical finite-range site percolation estimate to obtain exponential decay of the probability of a left--right crossing (which is standard technique in classical percolation theory \cite{DC}; see also for example the proof of Proposition~4.2 in \cite{DD18}) together with the Russo-Seymour-Welsh estimates (to come back to ${\bar{\ell}_n}(\psi,p)$), we have $$\Pro \left(L_{3 k,k}^{(n)}(\psi) \geq Ck^2 {\bar{\ell}_n}(\psi,p) \right) \leq Ce^{-c k}.$$
Therefore, using this bound together with Proposition \ref{Prop:ComparisonFields} to bound $X_{3k,k}$ (recalling the definition \eqref{DefX}),
\begin{align*}
\Pro \left(L_{3k,k}^{(n)}(\phi) \geq e^{\xi C \sqrt{k}} C_p C k^2 {\bar{\ell}_n}(\phi,p/2) \right) & \leq \Pro \left( e^{\xi X_{3k,k}} L_{3k,k}^{(n)}(\psi) \geq e^{\xi C \sqrt{k}} C_p C k^2 {\bar{\ell}_n}(\phi,p/2)  \right) \\
& \leq \Pro \left( X_{3k,k} \geq C \sqrt{k} \right) + \Pro \left( L_{3k,k}^{(n)}(\psi) \geq C_p C k^2 {\bar{\ell}_n}(\phi,p/2)  \right) \\
& \leq C e^{-ck} + \Pro \left( L_{3k,k}^{(n)}(\psi) \geq C k^2 {\bar{\ell}_n}(\psi,p)  \right) \leq C e^{-ck}.
\end{align*}
Note that we used the bound ${\bar{\ell}_n}(\psi,p) \leq C_p {\bar{\ell}_n}(\phi,p/2)$ from \eqref{eq:RatiosPsiPhi} in the third inequality; here $C_p$ is defined as in  \eqref{eq:RatiosPsiPhi}.

\textit{Step 2:} Decoupling and scaling. In this step, we give a rough bound of the coarse field $\phi_{0,m}$, to obtain spatial independence of the remaining field between blocks of size $2^{-m}$. When an event occurs on one block with high enough probability, the  percolation argument of Step 1 then provides, with very high probability, a left-right path of such events occuring simultaneously. Since $L_{3,1}^{(n)}(\phi) \leq e^{\xi \max_{R_{3,1}} \phi_{0,m}} L_{3,1}^{(m,n)}(\phi)$, the scaling property of the field $\phi$, i.e. $L_{3,1}^{(m,n)}(\phi) \overset{(d)}{=} 2^{-m} L_{3 \cdot 2^m,2^m}^{(n-m)}(\phi)$, 
gives 
\begin{align*}
\Pro& \left( L_{3,1}^{(n)}(\phi) \geq e^{\xi s \sqrt{m}} e^{c \sqrt{2^m}} \bar{\ell}_{n-m}(\phi,p) \right)  \\&\leq \Pro \left( \max_{R_{3,1}} \phi_{0,m} \geq Cm + s \sqrt{m} \right) + \Pro \left(2^{-m} L_{3 \cdot 2^m,2^m}^{(n-m)}(\phi) \geq e^{c \sqrt{2^m}} \bar{\ell}_{n-m}(\phi,p)  \right) 
\leq Ce^{-c s^2} + C e^{-c 2^m},
\end{align*}
where the first term of the second expression is bounded by taking $a = C + s m^{-1/2}$ in Proposition \ref{Prop:MaxBound} and the second bound follows from the result obtained in Step 1 with $k = 2^{m}$, taking a slightly larger $c$ in $\exp(c \sqrt{2^m})$ to absorb the factor $e^{Cm}$.

\textit{Step 3:} We derive an a priori bound ${\ell_n}(\phi,p) \geq 2^{-2\xi k} \ell_{n-k}(\phi,p) e^{-C \sqrt{k}}$. (Note that the argument below will be optimized in \eqref{eq:WeakSuperMul}.) For each dyadic block of size $2^{-k}$ visited by $\pi_n(\phi)$, one of the four rectangles of size $2^{-k}(1,3)$ around $P$ has to be crossed by $\pi_n(\phi)$. Therefore, since $\pi_n(\phi)$ has to visit at least $2^k$ dyadic blocks of size $2^{-k}$,  we have
$$
L_{1,1}^{(n)}(\phi) \geq 2^{k} e^{\xi \inf_{[0,1]^2} \phi_{0,k}} \min_{P \in \mc{P}_k, P \cap \pi_n(\phi) \neq \emptyset} \min_{1 \leq i \leq 4} L^{(k,n)}(R_i^S(P),\phi),
$$
where $(R_i^S(P))_{1 \leq i \leq 4}$ denote the four long rectangles of size $2^{-k}(1,3)$ surrounding $P$. Using the supremum tail estimate \eqref{eq:MaxBoundTail} and the left tail estimates \eqref{eq:LowerTailsPhi}, we get ${\ell_n}(\phi,p) \geq 2^{-2\xi k} \ell_{n-k}(\phi,p) e^{-C \sqrt{k}}$. Indeed,
\begin{align*}
& \Pro \left(  e^{\xi \inf_{[0,1]^2} \phi_{0,k}} \min_{P \in \mc{P}_k, P \cap \pi_n(\phi) \neq \emptyset} \min_{1 \leq i \leq 4} 2^{k} L^{(k,n)}(R_i^S(P),\phi)  \leq  2^{-2\xi k} \ell_{n-k}(\phi,p) e^{-C \sqrt{k}} \right) \\
&\leq  \Pro \left( \inf_{[0,1]^2} \phi_{0,k} \leq - k \log 4 - C \sqrt{k}  \right) + \Pro \left( \min_{P \in \mc{P}_k, P \cap \pi_n(\phi) \neq \emptyset} \min_{1 \leq i \leq 4} 2^{k} L^{(k,n)}(R_i^S(P),\phi)  \leq  \ell_{n-k}(\phi,p) e^{-C \sqrt{k}} \right) 
\end{align*}
and each term is less than $p/2$ if $C$ is large enough, depending on $p$.
Therefore, we have \[\bar{\ell}_{n-m}(\phi,p) \leq \Lambda_{n-m}(\phi,p) {\ell_{n-m}}(\phi,p) \leq  2^{2\xi m} e^{C \sqrt{m}} \Lambda_{n-m}(\phi,p) {\ell_n}(\phi,p).\] Now, by coming back to the partial result obtained in Step 2 and by taking $s^2 = 2^m$ for $s \in [1,2^{n/2}]$, we get
$$
\Pro \left( L_{3,1}^{(n)}(\phi) \geq e^{c s \sqrt{\log s}} e^{cs} \Lambda_n(\phi,p) {\ell_n}(\phi,p) \right)  \leq e^{- c s^2}.
$$

\textit{Step 4:} Now we consider large tails, so we assume $s \geq 2^{\frac{n}{2}}$. By a direct comparison with the supremum, we have ${\ell_n}(\phi,p) \geq 2^{-\xi (2 n + C \sqrt{n})}$ (later on we will use a more precise estimate from \cite{DG18}, see \eqref{eq:DGlowerBound}). Moreover, bounding from above the left-right distance by taking a straight path from left to right and then using a moment method analogous to the one in \eqref{eq:MomentEst}, we get $\Pro \left( L_{1,1}^{(n)}(\phi) \geq e^{\xi s} \right) \leq e^{-\frac{s^2}{2 (n+1) \log 2}}$. Altogether,
\[
\Pro \left( L_{1,1}^{(n)}(\phi) \geq {\ell_n}(\phi,p) \Lambda_n(\phi,p) e^{\xi s} \right) \leq \Pro \left( L_{1,1}^{(n)}(\phi) \geq {\ell_n}(\phi,p) e^{\xi s} \right) \leq e^{-\frac{(s-n \log 4 - C \sqrt{n})^2}{2(n+1)\log 2}} \leq e^{C s} e^{-c \frac{s^2}{\log s}},
\]
where we used $\Lambda_n(\phi,p) \geq 1$ in the first inequality and the bound ${\ell_n}(\phi,p) \geq 2^{-\xi (2 n + C \sqrt{n})}$ together with the tail estimate $\Pro \left( L_{1,1}^{(n)}(\phi) \geq e^{\xi s} \right) \leq e^{-\frac{s^2}{2 (n+1) \log 2}}$ in the second one. The last inequality follow since $s \geq 2^{\frac{n}{2}}$.

Combining the tail estimate of Step 3, valid for $s \in [1,2^{n/2}]$, and the one of Step 4, valid for $s \geq 2^{n/2}$, completes the proof.
\end{proof}

Using again the comparison between $\phi$ and $\psi$ given in Proposition \ref{Prop:ComparisonFields}, we get the following corollary.
\begin{Cor}[Upper tail estimates for $\psi$] For $p$ small enough, but fixed, we have, for all $n \geq 0$ and $s >2$,
\begin{equation}
\label{eq:UpperTailsPsi}
\Pro \left( L_{3,1}^{(n)}(\psi) \geq e^{s} \Lambda_{n}(\psi,p) {\ell_n}(\psi,p)  \right) \leq C e^{c \frac{s^2}{\log s}}.
\end{equation}
\end{Cor}

\section{Concentration}

\subsection{Concentration of the log of the left-right crossing length}

\paragraph{Condition (T).} Denote by $\pi_n(\psi)$ the left-right geodesic of the unit square associated to the field $\psi_{0,n}$. If there are multiple such geodesics, let $\pi_n(\psi)$ be chosen among them in some measurable way, for example by taking the uppermost geodesic. By $\pi_n^K(\psi)$ its $K$-coarse graining which we define as
\begin{equation}
\label{DefCoarseGraining}
\pi_n^K(\psi) := \lbrace P \in \mc{P}_K ~ : ~ P \cap \pi_n(\psi) \neq \varnothing \rbrace,
\end{equation}
recalling the definition \eqref{Pndef} of $\mc{P}_K$. Let $\psi_{0,n}(P)$ denote the value of the field $\psi_{0,n}$ taken at the center of a block $P$. We introduce the following condition: there exist constants $\alpha > 1$, $c > 0$ so that for $K$ large we have
$$
\label{eq:AssA}
\sup_{n \geq K} \E \left( \left( \frac{\sum_{P \in \pi_n^K(\psi)} e^{2\xi \psi_{0,K}(P)}}{\left( \sum_{P \in \pi_n^K(\psi)} e^{\xi \psi_{0,K}(P)}  \right)^2} \right)^{\alpha} \right)^{1/\alpha} \leq e^{-c K}.  \qquad \text{(Condition (T))}
$$
The importance of Condition (T) comes from the following theorem. 
\begin{thm}
\label{thm:AssTthm}
If $\xi$ is such that Condition (T) above is satisfied, then $(\log L_{1,1}^{(n)}(\phi) - \log \lambda_n(\phi))_{n \geq 0}$ is tight, where $\lambda_n(\phi)$ denotes the median of $L_{1,1}^{(n)}$.
\end{thm}

It is not expected that the weight is approximately constant over the crossing (since there may be some large level lines of the field that the crossing must cross). Condition~(T), however, roughly requires that the length of the crossing is supported by a number of coarse blocks that grows at least like some small but positive power of the total number of coarse blocks. Note that the fraction in Condition~(T) is the $\ell^2$ norm of the vector of crossing weights on each block divided by the square of the $\ell^1$ norm of the same, and thus controlling it amounts to an anticoncentration condition for this vector.

The core of this section is the proof of Theorem \ref{thm:AssTthm}. Before proving it, let us already jump to the important following proposition. Here we use the assumption that $\xi\in(0,2/d_2)$, although the formulation of Condition (T) is designed so that it could also hold for larger $\xi$.

\begin{Prop}
\label{Prop:CondSatisfied}
If $ \gamma \in (0,2)$, then $\xi := \frac{\gamma}{d_{\gamma}}$ satisfies Condition (T).  
\end{Prop}

\begin{proof}
\textit{Step 1:} Supremum bound. Taking the supremum over all blocks of size $2^{-K}$ in $[0,1]^2$, we get
$$
\frac{\sum_{P \in \pi_n^K(\psi)} e^{2\xi \psi_{0,K}(P)}}{\left( \sum_{P \in \pi_n^K(\psi)} e^{\xi \psi_{0,K}(P)}  \right)^2}  \leq \frac{e^{\xi \max_{P \in \mc{P}_K} \psi_{0,K}(P)}}{\sum_{P \in \pi_n^K(\psi)} e^{\xi \psi_{0,K}(P)}}  \leq \frac{e^{\xi \max_{P \in \mc{P}_K} \phi_{0,K}(P)}}{\sum_{P \in \pi_n^K(\psi)} e^{\xi \psi_{0,K}(P)}}  e^{\xi X_1},
$$
recalling the definition of $X_1$ below $\eqref{DefX}$.

\textit{Step 2:} We give a lower bound of the denominator of the right-hand side. By taking the concatenation of straight paths in each box of $\pi_n^K(\psi)$, we get a left-right crossing of $[0,1]^2$. Denote this crossing by $\Gamma_{n,K,\psi}$. We have,
\begin{multline}
\label{eq:CoarseToPath}
\sum_{P \in \pi_n^K(\psi)} e^{\xi \psi_{0,K}(P)} \geq e^{-\xi X_1} \sum_{P \in \pi_n^K(\psi)} e^{\xi \phi_{0,K}(P)} \\
 \geq e^{-\xi X_1} \exp(- \xi \max_{P \in \mc{P}_K^1} \osc_P(\phi_{0,K})) 2^{K} L^{(K)}(\phi,\Gamma_{n,K,\psi})   \geq e^{-\xi X_1}  \exp(- \xi \max_{P \in \mc{P}_K^1} \osc_P(\phi_{0,K})) 2^{K} L_{1,1}^{(K)}(\phi),
\end{multline}
where $\osc_P$ was defined in \eqref{Def:Osc} and $\mc{P}_K^1$ was defined in \eqref{eq:PK1}.

\textit{Step 3:} Combining the two previous steps, we have
$$
\frac{\sum_{P \in \pi_n^K(\psi)} e^{2\xi \psi_{0,K}(P)}}{\left( \sum_{P \in \pi_n^K(\psi)} e^{\xi \psi_{0,K}(P)}  \right)^2}  \leq \frac{e^{\xi \max_{P \in \mc{P}_K^1} \phi_{0,K}(P)}}{2^{K} L_{1,1}^{(K)}(\phi)} e^{2 \xi X_1} e^{ \xi \max_{P \in \mc{P}_K^1} \osc_P(\phi_{0,K})}.
$$
Now, we take $\alpha > 1$ close to $1$. Using H\"older's inequality with $\frac{1}{r} + \frac{1}{s}  = 1$ and $r$ close to $1$, together with Cauchy-Schwarz, we get
\begin{align*}
& \E \left( \left( \frac{\sum_{P \in \pi_n^K(\psi)} e^{2\xi \psi_{0,K}(P)}}{\left( \sum_{P \in \pi_n^K(\psi)} e^{\xi \psi_{0,K}(P)}  \right)^2}  \right)^{\alpha} \right)^{1/\alpha} \leq 2^{-K} \E \left(  \frac{e^{\alpha \xi \max_{P \in \mc{P}_K} \phi_{0,K}(P)}}{(L_{1,1}^{(K)}(\phi))^{\alpha}} e^{2 \alpha \xi X_1} e^{\alpha \xi \max_{P \in \mc{P}_K^1} \osc_P(\phi_{0,K})} \right)^{1/\alpha} \\
& \leq 2^{-K} \E \left( e^{\alpha r \xi \max_{P \in \mc{P}_K^1} \phi_{0,K}(P)} \right)^{1/\alpha r} \E \left( \left( L_{1,1}^{(K)}(\phi) \right)^{-2 \alpha s} \right)^{1/2 \alpha s} \E \left( e^{8 \alpha s \xi X_1} \right)^{1/4 \alpha s} \E \left( e^{4 \alpha s \xi \max_{P \in \mc{P}_K^1} \osc_P(\phi_{0,K})} \right)^{1/4 \alpha s}.
\end{align*}
Therefore, using \eqref{eq:MaxBoundExp} for the maximum, \eqref{eq:LowerTailsPhi} for the left-right crossing, Proposition \ref{Prop:ComparisonFields} to bound $X_1$ and  \eqref{eq:OscBoundExp} for the maximum of oscillations, we finally get, when $\alpha r \xi < 2$ (recall that $\alpha r$ can be taken arbitrarily close to $1$), 
\begin{equation}
\label{eq:FirstBound}
\E \left( \left( \frac{\sum_{P \in \pi_n^K(\psi)} e^{2\xi \psi_{0,K}(P)}}{\left( \sum_{P \in \pi_n^K(\psi)} e^{\xi \psi_{0,K}(P)}  \right)^2}  \right)^{\alpha} \right)^{1/\alpha}  \leq 2^{-K} 2^{2 \xi K} \ell_{1,1}^{(K)}(\phi,p)^{-1} e^{ C \sqrt{K}}.
\end{equation}

\textit{Step 4:} Lower bound on quantiles. For $\gamma \in (0,2)$, $Q := \frac{2}{\gamma} + \frac{\gamma}{2} > 2$. Using Proposition 3.17 from \cite{DG18} (circle average LFPP) and Proposition 3.3 from \cite{DGo18} (comparison between $\phi_{\delta}$ and circle average), we have, if $p$ is fixed and $\eps \in (0,Q-2)$, for $K$ large enough, 
\begin{equation}
\label{eq:DGlowerBound}
\ell_{1,1}^{(K)}(\phi,p) \geq 2^{-K(1-\xi Q + \xi \eps)}.
\end{equation}

\textit{Step 5:} Conclusion. Using the results from the two previous steps, we finally get
$$
\E \left( \left( \frac{\sum_{P \in \pi_n^K(\psi)} e^{2\xi \psi_{0,K}(P)}}{\left( \sum_{P \in \pi_n^K(\psi)} e^{\xi \psi_{0,K}(P)}  \right)^2}  \right)^{\alpha} \right)^{1/\alpha}  \leq 2^{- \xi (Q - 2 - \eps) K} e^{ C \sqrt{K}},
$$
which completes the proof.
\end{proof}

Now, we come back to the proof of Theorem \ref{thm:AssTthm}. We first derive a priori estimates on the quantile ratios.

\begin{Lem}
\label{Lem:Quantiles/Var}
Let $Z$ be a random variable with finite variance and $p \in (0,1/2)$. If a pair $(\bar{\ell}(Z,p),\ell(Z,p) )$ satisfies $\bar{\ell}(Z,p) \geq \ell(Z,p)$,  $\Pro ( Z \geq \bar{\ell}(Z,p) ) \geq p$ and $ \Pro ( Z \leq \ell(Z,p) ) \geq p$, then, we have:
\begin{equation}
(\bar{\ell}(Z,p) - \ell(Z,p))^2 \leq \frac{2}{p^2} \Var Z.
\end{equation}
\end{Lem}

\begin{proof}
If $Z'$ is an independent copy of $Z$, notice that for $l' \geq l$ we have $2 \mathrm{Var}(Z)  = \mathbb{E}( ( Z'-Z )^2 )  \geq \mathbb{E} ( 1_{Z' \geq l'} 1_{Z \leq l} (Z' -Z )^2 ) \geq  \Pro ( Z \geq l' ) \Pro ( Z \leq l ) ( l' - l )^2$.
\end{proof}

In the following lemma, we derive an a priori bound on the variance of $\log L_{1,1}^{(n)}(\phi)$.

\begin{Lem}
\label{Lem:VarApriori}
For all $n \geq 0$ we have the bound
$$
\Var \log L_{1,1}^{(n)}(\phi) \leq \xi^2 (n+1) \log 2
$$
\end{Lem}

\begin{proof}
Denote by $L_{1,1}^{(n)}(\mathrm{D}_k)$ the left-right distance of $[0,1]^2$ for the length metric $e^{\xi \phi_{0,n}^k} ds$, where $\phi_{0,n}^k$ is piecewise constant on each dyadic block of size $2^{-k}$ where it is equal to the value of $\phi_{0,n}$ at the center of this block. (We do not assign an independent meaning to the notation $\mathrm{D}_k$.) Note that we have
$$
e^{-C 2^{-k} \norme{\con \phi_{0,n}}_{[0,1]^2} } L_{1,1}^{(n)} \leq L_{1,1}^{(n)}(\mathrm{D}_k) \leq L_{1,1}^{(n)} e^{C 2^{-k} \norme{\con \phi_{0,n}}_{[0,1]^2}},
$$
which gives almost surely that $L_{1,1}^{(n)}(\phi) = \lim_{k \to \infty} L_{1,1}^{(n)}(\mathrm{D}_k)  $. By dominated convergence we have
$$
\Var \log L_{1,1}^{(n)}(\phi)  = \lim_{k \to \infty} \Var \log L_{1,1}^{(n)}(\mathrm{D}_k).
$$
Now, $\log L_{1,1}^{(n)}(\mathrm{D}_k)$ is a $\xi$-Lipschitz function of $p=4^k$ Gaussian variables denoted by $Y = (Y_1, . . . , Y_p)$, where on $\R^p$ we use the supremum metric. We can write $Y = A N$ for some symmetric positive semidefinite matrix $A$ and standard Gaussian vector $N$ on $\mb{R}^{4^k}$. Then $\log L_{1,1}^{(n)}(\mathrm{D}_k) = f(Y) = f(A N)$ which is $\xi \sigma$-Lipschitz as a function of $N$ where $\sigma = \max(|A_1|, . . . , |A_p|)$. By the Gaussian concentration inequality of \cite[Lemma 2.1]{DZZ18}, applied as in \cite[Lemma 5.8]{DD18}, since the pointwise variance of the field is $(n+1) \log 2$ we have
\begin{equation*}
\Var \log L_{1,1}^{(n)}(\mathrm{D}_k) \leq \max(\Var(Y_1), . . . , \Var(Y_p)) = \xi^2 (n+1) \log 2.\qedhere
\end{equation*}\end{proof}
Before stating the following lemma, we refer the reader to the definition of quantile ratios in  \eqref{DefQuant}.
\begin{Lem}[A priori bound on the quantile ratios]
\label{Lem:Apriori}
Fix $p \in (0,1/2)$. There exists a constant $C_p$ depending only on $p$ such that for all $n \geq 1$,
\begin{equation}
\label{eq:apriori}
\Lambda_n(\psi,p) \leq e^{C_p \sqrt{n}}.
\end{equation}
\end{Lem}

\begin{proof}
By using Lemma \ref{Lem:VarApriori} we get $\Var(\log L_{1,1}^{(k)}(\psi)) \leq C k$ for all $1 \leq k \leq n$ and an absolute constant $C > 0$. This implies the same bound for $\psi$ by Proposition~\ref{Prop:ComparisonFields}. Using then Lemma \ref{Lem:Quantiles/Var} with $Z_k = \log L_{1,1}^{(k)}(\psi)$ for $k \leq n$, we finally get the bound $\max_{k \leq n} \frac{ \bar{\ell}_k(\psi,p)}{{\ell_k}(\psi,p)} \leq e^{C_p \sqrt{n}}$.
\end{proof}

\begin{proof}[Proof of Theorem \ref{thm:AssTthm}] 

The proof is divided in five steps. $K$ will denote a large positive number to be fixed at the last step.

\textit{Step 1.} Quantiles-variance relation / setup. We aim to get an inductive bound on $\Lambda_n(\psi,p)$. We will therefore bound $\frac{{\bar{\ell}_n}(\psi,p/2)}{{\ell_n}(\psi,p/2)}$ in term of $\Lambda$'s at lower scales. $p$ will be fixed from now on, small enough so that we have the tail estimates from Section \ref{Sec:TailEstimates} for $\phi$ with $p$ and for $\psi$ with $p/2$. The starting point is the bound
\begin{equation}
\label{eq:Step1}
\frac{{\bar{\ell}_n}(\psi,p/2)}{{\ell_n}(\psi,p/2)} \leq e^{C_p \sqrt{\Var \log L_{1,1}^{(n)}(\psi)}}.
\end{equation}

\textit{Step 2.} Efron-Stein. Using the Efron-Stein inequality with the block decomposition of $\psi_{0,n}$ introduced in \eqref{eq:BlockDecompoPsi}, defining the length with respect to the unresampled field $L_n(\psi) = L^{(n)}_{1,1}(\psi)$, we get
\begin{equation}
\label{eq:EfronStein}
\Var \log L_{1,1}^{(n)}(\psi) \leq   \E \left( \left( \log L_n^K(\psi) - \log L_n(\psi) \right)_{+}^2  \right)  + \sum_{P \in \mc{P}_K} \E \left( \left( \log L_n^P(\psi) - \log L_n(\psi) \right)_{+}^2 \right),
\end{equation}
where in the first term (resp. second term) we resample the field $\psi_{0,K}$ (resp. $\psi_{K,n,P}$) to get an independent copy $\tilde{\psi}_{0,K}$ (resp. $\tilde{\psi}_{K,n,P}$) and we consider the left-right distance $L_n^K(\psi)$ (resp. $L_n^P(\psi)$) of the unit square associated to the field $\psi_{0,n} - \psi_{0,K} + \tilde{\psi}_{0,K}$ (resp. $\psi_{0,n} - \psi_{K,n,P} + \tilde{\psi}_{K,n,P}$).

\textit{Step 3.} Analysis of the first term. For the first term, using Gaussian concentration as in the proof of Lemma \ref{Lem:VarApriori}, we get
\begin{equation}
\label{eq:SndTerm}
\E((\log L_n^K(\psi)- \log L_n(\psi))^2) = 2 \E (\Var (\log L_n(\psi) | \psi_{0,n} - \psi_{0,K})) \leq C K.
\end{equation}

\textit{Step 4.} Analysis of the second term. For $P \in \mc{P}_K$, if $L_n^P(\psi) > L_n(\psi)$, the block $P$ is visited by the geodesic $\pi_n(\psi)$ associated to $L_n(\psi)$. Define
\begin{equation}
\label{Def:PK}
P^K := \lbrace Q \in \mc{P}_K ~ : ~ d(P,Q) \leq C K^{\eps_0} 2^{-K} \rbrace.
\end{equation}
where we recall that $\eps_0$ is associated with the range of dependence of the resampled field $\tilde{\psi}_{K,n,P}$ through \eqref{Def:sigma} (see also the subsection following this definition).  Here, $d(P,Q)$ is the $L^{\infty}$-distance between the sets $P$ and $Q$.

We upper-bound $L_n^P(\psi)$ by taking the concatenation of the part of $\pi_n(\psi)$ outside of $P^{K}$ together with four geodesics associated to long crossings in rectangles comprising a circuit around $P^K$ (for the field $\psi_{0,n}$ which coincides with the field $\psi_{0,n}^P$ outside of $P^{K}$). We get, introducing the rectangles $(Q_{i}(P))_{1 \leq i \leq 4}$ of size $2^{-K}(CK^{\eps_0},3)$ surrounding $P^K$ ($P^K$ and its $3 \cdot 2^{-K}$ neighborhood form an annulus, and gluing the four crossings gives a circuit in this annulus) and using the inequality $\log x \leq x -1$,
\begin{equation}
\label{eq:Ineq}
\left( \log L_n^P(\psi) - \log L_n(\psi) \right)_{+} \leq \frac{( L_{n}^P(\psi) - L_n(\psi) )_{+} }{L_{1,1}^{(n)}(\psi)} \leq 4 \frac{\max_{1 \leq i \leq 4}L^{(n)}(Q_i(P),\psi)}{L_{1,1}^{(n)}(\psi)}.
\end{equation}

$\bullet$ We recall the notation $\phi_{0,K}(P)$ to denote the value of the field $\phi_{0,K}$ at the center of $P$. We bound from above each term in the maximum of \eqref{eq:Ineq} as follows:
\begin{align*}
L^{(n)}(Q_i(P),\psi) & \leq e^{\xi X} L^{(n)}(Q_i(P),\phi) \\
& \leq e^{\xi X} e^{\xi \phi_{0,K}(P)} e^{\xi \osc_{P^K} (\phi_{0,K})} L^{(K,n)}(Q_i(P),\phi) \\
& \leq  e^{2\xi X} e^{\xi \psi_{0,K}(P)} e^{\xi \osc_{P^K} (\phi_{0,K})} L^{(K,n)}(Q_i(P),\phi),
\end{align*}
where the oscillation $\osc$ is defined in $\eqref{Def:Osc}$ and $P^K$ is defined in $\eqref{Def:PK}$. 

For a rectangle $Q$ of size $2^{-K}$, with corners in $2^{-K} \Z^2$, we denote by  $(R_{i}^L(Q))_{1 \leq i \leq 4}$ the four long rectangles of size $2^{-K}(3,1)$ surrounding $Q$. We can upper-bound the rectangle crossing lengths associated to the $Q_i(P)$'s by gluing $O(K^{\eps_0})$ rectangle crossings of size $2^{-K}(3,1)$, which include an annulus around each block $Q$ of size $2^{-K}(1,1)$ (with corners in $2^{-K}\Z^2$) in the shaded region $A^K$ of Figure \ref{fig:Step4}. We get
$$
\max_{1 \leq i \leq 4}L^{(K,n)}(Q_i(P),\phi) \leq C K^{\eps_0} \max_{Q \in A^K, 1\leq i \leq 4} L^{(K,n)}(R_{i}^L(Q),\phi)
$$
and we end up with the following upper bound:
\begin{equation}
\label{eq:Num}
\left( \log L_n^P(\psi) - \log L_n(\psi) \right)_{+} \leq   e^{2\xi X} \frac{e^{\xi \psi_{0,K}(P)}}{L_{1,1}^{(n)}(\psi)} e^{\xi \osc_{P^K} (\phi_{0,K})}  C K^{\eps_0} \max_{Q \in A^K, 1\leq i \leq 4} L^{(K,n)}(R_{i}^L(Q),\phi).
\end{equation}

\begin{figure}[ht]
\centering
\includegraphics[scale=1]{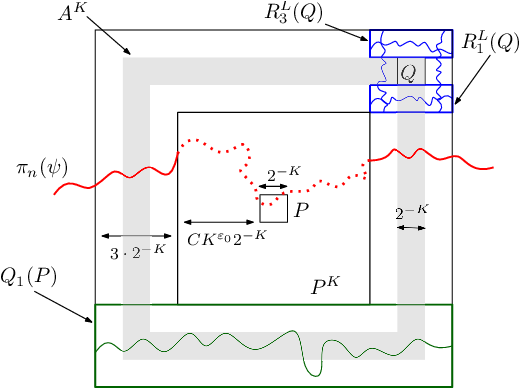}
\caption{Illustration of the geodesics used in the upper bound of Step 4.} 
\label{fig:Step4}
\end{figure}

$\bullet$ We lower-bound the denominator of \eqref{eq:Num} as follows. If $P \in \mc{P}_K$ is visited by a $\pi_{n}(\psi)$ geodesic, then there are at least two short disjoint rectangle crossings among the four surrounding $P$. Therefore, if we denote by $\hat{P}$ the box containing $P$ at its center whose size is three times that of $P$,
\begin{align*}
\int_{\pi_{n}(\psi) \cap \hat{P}} e^{\xi \psi_{0,n}}  ds & \geq 2 \min_{1 \leq i \leq 4} L^{(n)}(R_i^S(P),\psi) \geq e^{-\xi X}  \min_{1 \leq i \leq 4} L^{(n)}(R_i^S(P),\phi) \\
& \geq e^{-\xi X} e^{\xi \phi_{0,K}(P)} e^{-\xi \osc_{\hat{P}}(\phi_{0,K})}   \min_{1 \leq i \leq 4} L^{(K,n)}(R_i^S(P),\phi) \\
& \geq e^{-2\xi X} e^{\xi \psi_{0,K}(P)} e^{-\xi \osc_{\hat{P}}(\phi_{0,K})}  \min_{1 \leq i \leq 4} L^{(K,n)}(R_i^S(P),\phi),
\end{align*}
where $(R_{i}^S(P))_{1 \leq i \leq 4}$ denote the four short rectangles of size $2^{-K}(1,3)$ surrounding $P$. Summing over all $P$'s and taking uniform bounds for the rectangle crossings at higher scales, 
\begin{align*}
L_{1,1}^{(n)}(\psi) & = \sum_{P \in \mc{P}_{K}} \int_{P \cap \pi_{n}(\psi)} e^{\xi \psi_{0,n}} ds  
 \geq \frac{1}{9}  \sum_{P \in \mc{P}_{K}} \int_{\hat{P} \cap \pi_{n}(\psi)} e^{\xi \psi_{0,n}} ds  \\
& \geq \frac{1}{9} e^{-2 \xi X} \left(  \min_{P \in \mc{P}_K^1}  \min_{1 \leq i \leq 4} L^{(K,n)}(R_i^S(P),\phi) \right)  \left( \sum_{P \in \mc{P}_{K}, P \cap \pi_{n}(\psi) \neq \emptyset}  e^{\xi \psi_{0,K}(P)} e^{- \xi \osc_{\hat{P}}(\phi_{0,K})} \right).
\end{align*}
Therefore, taking a uniform bound for the oscillation, we get
\begin{align}
L_{1,1}^{(n)}(\psi) & \geq  \frac{1}{9} e^{-2 \xi X} \left( \sum_{P \in \pi_n^K(\psi)} e^{\xi \psi_{0,K}(P)} e^{-\xi \osc_{\hat{P}}(\phi_{0,K})} \right) \min_{P \in \mc{P}_K^1, 1\leq i \leq 4} L^{(K,n)}(R_i^S(P),\phi) \\ 
& \label{eq:Denum} \geq  \frac{1}{9} e^{-2 \xi X } e^{-\xi \max_{P \in \mc{P}_K^1} \osc_{\hat{P}}(\phi_{0,K})}  \min_{P \in \mc{P}_K^1, 1\leq i \leq 4} L^{(K,n)}(R_i^S(P),\phi)  \sum_{P \in \pi_n^K(\psi)} e^{\xi \psi_{0,K}(P)}.
\end{align}

$\bullet$ We recall that $(R_{i}^L(P))_{1 \leq i \leq 4}$ denote the four rectangles of size $2^{-K}(3,1)$ surrounding $P$. Gathering inequalities \eqref{eq:Num} and \eqref{eq:Denum}, we have
\begin{multline*}
\sum_{P \in \mc{P}_K} \E \left( \left( \log L_n^P(\psi) - \log L_n(\psi) \right)_{+}^2 \right) \\
\leq  C K^{2\eps_0} \E \left( \frac{\sum_{P \in \pi_n^K(\psi)} e^{2\xi \psi_{0,K}(P)}}{\left( \sum_{P \in \pi_n^K(\psi)} e^{\xi \psi_{0,K}(P)}  \right)^2}  \left( \frac{ \max_{P \in \mc{P}_K^1, 1\leq i \leq 4} L^{(K,n)}(R_i^L(P),\phi)}{ \min_{P \in \mc{P}_K^1, 1\leq i \leq 4} L^{(K,n)}(R_i^S(P),\phi)} \right)^2  e^{C \xi \max_{P \in \mc{P}_K^1} \osc_{P^K} (\phi_{0,K})} e^{8 \xi X }\right).
\end{multline*}

$\bullet$ Condition (T) gives us a $\alpha > 1$ and $c >0$ so that for $K$ large enough, for $n \geq K$,
$$
\E \left( \left( \frac{\sum_{P \in \pi_n^K(\psi)} e^{2\xi \psi_{0,K}(P)}}{\left( \sum_{P \in \pi_n^K(\psi)} e^{\xi \psi_{0,K}(P)}  \right)^2} \right)^{\alpha} \right)^{1/\alpha} \leq e^{-c K}.
$$
Then, by using the gradient estimate \eqref{eq:OscBoundExp} and recalling the definition of $\mc{P}^{K}$ in \eqref{Def:PK}, we have
\begin{equation}
\label{eq:EffectE0}
\E \left( e^{C \max_{P \in \mc{P}_K^1} \osc_{P^K}(\phi_{0,K})} \right) \leq \E \left( e^{C K^{\eps_0} 2^{-K} \norme{ \con \phi_{0,K} }_{[0,1]^2}} \right) \leq e^{C K^{\frac{1}{2} + \eps_0}}.
\end{equation}
It is for the second inequality that in \eqref{Def:sigma} we take $\eps_0$ to be small in the definition of $\psi$;  $\eps_0 < 1/2$ is sufficient. Furthermore, using our tail estimates with regard to upper and lower quantiles for $\phi$ (see \eqref{eq:LowerTailsPhi} and \eqref{eq:UpperTailsPhi}, and the scaling property \eqref{eq:scaling-length}, for $\beta > 1$ so that $\frac{1}{\alpha}+\frac{1}{\beta} = 1$, we get
\begin{equation}
\label{eq:FirstIneq}
\E \left( \left( \frac{ \max_{P \in \mc{P}_K^1, 1\leq i \leq 4} L^{(K,n)}(R_i^L(P),\phi)}{\min_{P \in \mc{P}_K^1, 1\leq i \leq 4} L^{(K,n)}(R_i^S(P),\phi)} \right)^{2\beta} \right)^{\frac{1}{\beta}} \leq \Lambda_{n-K}^2 (\phi,p) e^{C K^{\frac{1}{2} + \eps_0}}.
\end{equation}
Note that we could have a $\log K$ term instead of the $K^{\eps_0}$ in \eqref{eq:FirstIneq}. Altogether, by applying H\"older inequality and Cauchy-Schwarz, we get
\begin{equation}
\label{eq:FstTerm}
\sum_{P \in \mc{P}_K} \E \left( \left( \log L_n^P(\psi) - \log L_n(\psi) \right)_{+}^2 \right) \leq e^{-cK} e^{C K^{\frac{1}{2}+\eps_0}} \Lambda_{n-K}^2(\phi,p) \leq e^{-cK} e^{C K^{\frac{1}{2} + \eps_0}}  C_p \Lambda_{n-K}^2(\psi,p/2),
\end{equation}
where we used \eqref{eq:RatiosPsiPhi} in the last inequality to get $ \Lambda_{n-K}^2(\phi,p)   \leq C_p \Lambda_{n-K}^2(\psi,p/2)$.

\textit{Step 5.} Conclusion. Gathering the bounds obtained in Step 3 (inequality \eqref{eq:SndTerm}) and Step 4 (inequality \eqref{eq:FstTerm}), we get, coming back to the inequality \eqref{eq:EfronStein}, for $K$ large enough,
\begin{equation}
\label{eq:RecursiveIneq}
\Var \log L_{1,1}^{(n)}(\psi) \leq C_1K + e^{-C_2 K} \Lambda_{n-K}^2(\psi,p/2).
\end{equation}
Now, we will show that this bound together with the a priori bound on the quantile ratios (Lemma \ref{Lem:Apriori}) is enough to conclude first that $\Lambda_{\infty}(\psi,p/2) < \infty$ and then that $\sup_{n \geq 0} \Var \log L_{1,1}^{(n)}(\psi) < \infty$, using the tail estimates \eqref{eq:LowerTailsPhi} and \eqref{eq:UpperTailsPsi}. 

Coming back to Step 1 (equation \eqref{eq:Step1}) and using \eqref{eq:RecursiveIneq}, we get the inductive inequality \eqref{Fin:ind} below for $K$ large enough and $n \geq K$, and  \eqref{Fin:ap} below by the a priori bound on the quantile ratios Lemma~\ref{Lem:Apriori}:
\begin{align}
\label{Fin:ind} \frac{{\bar{\ell}_n}(\psi,p/2)}{{\ell_n}(\psi,p/2)} & \leq e^{C_p \sqrt{\Var \log L_{1,1}^{(n)}(\psi)}} \leq e^{C_p \sqrt{C_1K + e^{-C_2 K} \Lambda_{n-K}^2(\psi,p/2)}}; \\
\label{Fin:ap}\Lambda_K(\psi,p/2)  &\leq e^{\tilde{C}_p \sqrt{K}}.
\end{align} 
From now on, we take $K$ large enough but fixed so that
\begin{equation}
\label{Fin:fixK}
 e^{-C_2 K} (e^{\tilde{C}_p \sqrt{K}} +e^{C_p \sqrt{2 C_1 K}} )^2  \leq C_1 K.
\end{equation}
Set 
\begin{equation}
\label{Def:drec}
\Lambda_{\mathrm{Rec}} := \Lambda_K(\psi,p/2) \vee e^{C_p \sqrt{2 C_1 K}}.
\end{equation} so that $\Lambda_K(\psi,p/2) \leq \Lambda_{\mathrm{Rec}}$. This is the initialization of the induction. Now, assume that $\Lambda_{n-1}(\psi,p/2) \leq \Lambda_{\mathrm{Rec}}$. In particular, $\Lambda_{n-K}(\psi,p/2) \leq \Lambda_{\mathrm{Rec}}$ and using \eqref{Fin:ind}
$$
\frac{{\bar{\ell}_n}(\psi,p/2)}{{\ell_n}(\psi,p/2)} \leq e^{C_p \sqrt{C_1 K + e^{-C_2 K} \Lambda_{\mathrm{Rec}}^2}}
$$
The right-hand side is smaller than $e^{C_p \sqrt{2C_1 K}}$ and therefore than $\Lambda_{\mathrm{Rec}}$. Indeed, by \eqref{Def:drec}, \eqref{Fin:ap} and \eqref{Fin:fixK},
$$
e^{-C_2 K} \Lambda_{\mathrm{Rec}}^2 \leq e^{-C_2 K} (\Lambda_K(\psi,p/2) + e^{C_p \sqrt{2 C_1 K}} )^2 \leq e^{-C_2 K} (e^{\tilde{C}_p \sqrt{K}} +e^{C_p \sqrt{2 C_1 K}} )^2  \leq C_1 K.
$$
Therefore, 
$$
\Lambda_n(\psi,p/2) = \Lambda_{n-1}(\psi,p/2) \vee  \frac{{\bar{\ell}_n}(\psi,p/2)}{{\ell_n}(\psi,p/2)}  \leq \Lambda_{\mathrm{Rec}}.
$$
Therefore, $\Lambda_{\infty}(\psi,p/2) < \infty$ thus $\Lambda_{\infty}(\phi,p) < \infty$ and by the tail estimates \eqref{eq:LowerTailsPhi} and \eqref{eq:UpperTailsPhi}, the sequence $(\log L_{1,1}^{(n)}(\phi) - \log \lambda_n(\phi))_{n \geq 0}$ is tight.
\end{proof}

\subsection{Weak multiplicativity of the characteristic length and error bounds}

Henceforth, we will only consider the case $\xi = \frac{\gamma}{d_{\gamma}}$ for $\gamma \in (0,2)$ and the field $\phi_{0,n}$. All observables will be assumed to be taken with respect to $\phi$ and we will drop the additional notation used to differ between $\phi$ and $\psi$. In this case, we saw that there exists a fixed constant $C > 0$ so that for all $n \geq 0$, $\bar{\ell}_{3,1}^{(n)}(p) \leq C \bar{\ell}_{1,3}^{(n)}(p)$, $C^{-1} \ell_{3,1}^{(n)}(p) \leq  \ell_{1,3}^{(n)}(p)$ and with the tail estimates, $\E (L_{3,1}^{(n)}) \leq C \E (L_{1,3}^{(n)})$. All these characteristic lengths are uniformly comparable. We will take $\lambda_n$ to denote one of them, say the median of $L_{1,1}^{(n)}$.

In the next elementary lemma, we prove that a sequence satisfying a certain quantitative weak multiplicative property has an exponent, and we quantify the error.
\begin{Lem}
\label{Lem:RealAnalysis}
Consider a sequence of positive real numbers $(\lambda_n)_{n \geq 1}$. If there exists $C > 0$ such that for all $ n \geq 1$, $k \geq 1$ we have
\begin{equation}
\label{eq:WeakMulAs}
e^{-C \sqrt{k}} \lambda_n \lambda_k \leq \lambda_{n+k} \leq e^{C \sqrt{k}} \lambda_n \lambda_k,
\end{equation}
then there exists $\rho > 0$ such that $\lambda_{n} = \rho^{n + O(\sqrt{n})}$. 
\end{Lem}

\begin{proof}
We introduce the sequence $(a_n)_{n \geq 0}$ such that $\lambda_{2^{n+1}} = \left( \lambda_{2^{n}} \right)^2 e^{a_n}$. By iterating, we get 
$$
\lambda_{2^{n+1}} =  \left( \lambda_{2^n}\right)^2 e^{a_{n}}  =  \left( \lambda_{2^{n-1}} \right)^4 e^{2 a_{n-1} + a_n} = \dots = \lambda_1^{2^{n+1}} e^{2^n a_0 + 2^{n-1} a _1 + \dots + 2 a_{n-1} + a_n}.
$$
The condition \eqref{eq:WeakMulAs} gives that the sequence $\left(2^{-n/2} a_n  \right)_{n \geq 0}$ is bounded, therefore the series $\sum_{k \geq 0} \frac{a_k}{2^{k}}$ converges and $| \sum_{k \geq n} \frac{a_k}{2^k} | \leq  2 \  (\sup_{k \geq 0} 2^{-k/2} |a_k| )\ 2^{-n/2}$. In particular there exists $\rho > 0$ such that
$$
\lambda_{2^{n+1}} = e^{ 2^{n+1} \left(\log \lambda_1 + \frac{1}{2} \sum_{k=0}^n  \frac{a_k}{2^k} \right)} = e^{2^{n+1} \left( \log \lambda_1 + \frac{1}{2} \sum_{k = 0}^{\infty} \frac{a_k}{2^k} \right)} e^{ - 2^{n} \sum_{k \geq n+1} \frac{a_k}{2^k}} = \rho^{2^n} e^{O(2^{n/2})}.
$$
Now that we have the existence of an exponent, we prove the upper bound of Lemma \ref{Lem:RealAnalysis}. There exist $C_1, C_2 > 0$ such that we have the following upper bounds:
\begin{align}
\label{eq:IneDya}
& \lambda_{2^k} \leq \rho^{2^k} e^{C_1 2^{k/2}}, \\
\label{eq:RightMul}
& \lambda_{n+k} \leq \lambda_n \lambda_k e^{C_2 \sqrt{k}}.
\end{align}
Take $C_3$ large enough so that 
$
(C_1 + C_2)^2 + (C_1 + C_2)C_3 \leq C_3^2$ and $\lambda_1 \leq \rho e^{C_3}$.
We want to prove by induction that for all $n \geq 1$, $\lambda_n \leq \rho^n e^{C_3 \sqrt{n}}$. The assumption on $C_3$ implies that this holds for $n =1$. By induction (in a dyadic fashion), take $n \in [2^k, 2^{k+1})$. We decompose $n$ as $n = 2^k + n_k$ with $n_k \in [0,2^k)$. We have, by using \eqref{eq:RightMul}, \eqref{eq:IneDya} and the induction hypothesis,
$$
\lambda_n \leq \lambda_{2^k} \lambda_{n_k}  e^{C_2 2^{k/2}}  \leq ( \rho^{2^k} e^{C_1 2^{k/2}}  ) (\rho^{n_k} e^{C_3 \sqrt{n_k}}) e^{C_2 2^{k/2}}   = \rho^{n} e^{(C_1 + C_2) 2^{k/2} + C_3 \sqrt{n_k}} \leq \rho^n e^{C_3 \sqrt{n}},
$$
since by the assumption on $C_3$ we have
$$
\left( (C_1 + C_2) 2^{k/2} + C_3 \sqrt{n_k} \right)^2 = (C_1 + C_2)^2 2^k + (C_1 + C_2) C_3 2^{k/2} \sqrt{n_k}+ C_3^2 n_k  \leq C_3^2 (2^k + n_k) = C_3^2 n.
$$
The proof of the lower bound is similar.\end{proof}

In the next proposition we prove that the characteristic length $\lambda_n$ satisfies the weak multiplicativity property \eqref{eq:WeakMulAs} and we identify the exponent by using the results of \cite{DG18}.
\begin{Prop}
\label{Prop:WeakMul}
For $\xi$ satisfying Condition (T), there exists $C > 0$ such that for all $ n \geq 1$, $k \geq 1$ we have
\begin{equation}
\label{eq:WeakMul}
e^{-C \sqrt{k}} \lambda_n \lambda_k \leq \lambda_{n+k} \leq e^{C \sqrt{k}} \lambda_n \lambda_k.
\end{equation}
Furthermore, when $\gamma \in (0,2)$ and $\xi = \gamma/d_{\gamma}$, we have 
\begin{equation}
\label{eq:ExpoError}
\lambda_{n} = 2^{-n(1-\xi Q) + O(\sqrt{n})}.
\end{equation}
\end{Prop}

\begin{proof}
Let us assume first that \eqref{eq:WeakMul} holds. Then, by using Lemma \ref{Lem:RealAnalysis}, there exists $\rho > 0$ such that we have $\lambda_n = \rho^{n + O(\sqrt{n})}$. Similarly to \eqref{eq:DGlowerBound}, for each fixed small $\delta >0$, for $k$ large enough we have,
\begin{equation}
\label{eq:DGupperQuantile}
\lambda_k \leq 2^{-k(1-\xi Q-\delta)}.
\end{equation}
The proof of \eqref{eq:DGupperQuantile} follows the same lines as the one of \eqref{eq:DGlowerBound}. Combining \eqref{eq:DGupperQuantile} and \eqref{eq:DGlowerBound} we get $\rho = 2^{-(1-\xi Q)}$. Now, we prove that the characteristic length satisfies \eqref{eq:WeakMul}.

\textit{Step 1:} Weak submultiplicativity. Let $\pi_k$ be such that $L^{(k)}(\pi_k) = L_{1,1}^{(k)}$. If $P \in \mathcal{P}_k$ is visited by $\pi_k$, consider the  concatenation $S^{(k,n+k)}(P)$ of four geodesics for $e^{\xi \phi_{k,n+k}} ds$ associated to the rectangles of size $2^{-k} (3,1)$ surrounding $P$.  Each geodesic is in the long direction of its rectangle so that this concatenation is a circuit. By scaling, $\E (L^{(k,n+k)}(S^{(k,n+k)}(P))) = 2^{-k+2} \E(L_{3,1}^{(n)})$. Note that 
the collection $\pi_k^k(\phi) = \lbrace P \in \mc{P}_k : P \cap \pi_k \neq \varnothing \rbrace$ is measurable with respect to $\phi_{0,k}$, which is independent of $\phi_{k,n+k}$. Set $\Gamma_{k,n} := \bigcup_{P\in \pi_k^k(\phi)} S^{(k,n+k)}(P)$. Note that $\Gamma_{k,n}$ contains a left-right crossing of $[0,1]^2$ whose length is bounded above by
\begin{equation*}
L^{(n+k)}(\Gamma_{k,n})  = \sum_{P\in \pi_k^k(\phi)} L^{(n+k)}(S^{(k,n+k)}(P))
  \leq  \sum_{P\in \pi_k^k(\phi) } L^{(k,n+k)}(S^{(k,n+k)}(P)) e^{\xi \phi_{0,k}(P)} e^{\xi \osc_{\hat{P}}(\phi_{0,k} )},
\end{equation*}
where $\hat{P}$ denotes the box containing $P$ at its center whose side length is three times that of $P$. Since $L_{1,1}^{(n+k)} \leq L^{(n+k)}(\Gamma_{k,n})$, by independence we have
\begin{align*}
\E (L_{1,1}^{(n+k)}) & \leq 4 \E(L_{3,1}^{(n)}) \E \left(\sum_{P\in\pi_k^k(\phi)} 2^{-k} e^{\xi \phi_{0,k} (P)}  e^{\xi \osc_{\hat{P}}(\phi_{0,k} )} \right).
\end{align*}
If $P$ is visited, then one of the four rectangles of size $2^{-k}(1,3)$ in $\hat{P}$ surrounding $P$ contains a short crossing, denoted by $\tilde{\pi}_k(P)$ and we have
$$
\int_{\pi_k} e^{\xi \phi_{0,k}} 1_{\pi_k \cap \hat{P}} ds \geq L^{(k)}(\tilde{\pi}_k(P)) \geq 2^{-k} e^{\xi \inf_{\hat{P}} \phi_{0,k}} \geq 2^{-k} e^{\xi \phi_{0,k}(P)} e^{- \xi  \osc_{\hat{P}}( \phi_{0,k})},
$$
hence 
\begin{align*}
\sum_{P\in \pi_k^k(\phi)} 2^{-k} e^{\xi \phi_{0,k} (P)}  e^{\xi \osc_{\hat{P}}(\phi_{0,k} )} & \leq \sum_{P\in \pi_k^k(\phi)}  e^{2\xi  \osc_{\hat{P}}( \phi_{0,k})} \int_{\pi_k} e^{\xi \phi_{0,k}} 1_{\pi_k \cap \hat{P}} ds.
\end{align*}
Taking the supremum of the oscillation over all blocks, 
$$
\sum_{P\in\pi_k^k(\phi)}  e^{2\xi  \osc_{\hat{P}}( \phi_{0,k})} \int_{\pi_k} e^{\xi \phi_{0,k}} 1_{\pi_k \cap \hat{P}} ds \leq 9 e^{2 \xi 2^{-k}  \norme{\nabla \phi_{0,k}}_{[0,1]^2}} L_{1,1}^{(k)}.
 $$
 Altogether, by Cauchy-Schwarz, 
$$
\E (L_{1,1}^{(n+k)})  \leq 36 \E(L_{3,1}^{(n)}) \E((L_{1,1}^{(k)})^2)^{1/2}  \E(e^{4 \xi  2^{-k}  \norme{\nabla \phi_{0,k}}_{[0,1]^2}})^{1/2}.
$$ 
When $\xi$ satisfies Condition (T), by using the uniform bounds for quantile ratios together with the upper tail estimates \eqref{eq:UpperTailsPhi} and the gradient estimate \eqref{eq:OscBoundExp} we get $\lambda_{n+k} \leq e^{C \sqrt{k}} \lambda_n \lambda_k $.

\textit{Step 2:} Weak supermultiplicativity. We argue here that 
\begin{equation}
\label{eq:WeakSuperMulQuantile}
\lambda_{n+k} \geq e^{-C\sqrt{k}} \lambda_{n} \lambda_k.
\end{equation}
Using a slightly easier argument than \eqref{eq:Denum} (since we just have the field $\phi$ here), we have
$$
L_{1,1}^{(n+k)} \geq e^{- \xi \max_{P \in \mc{P}_k^1} \osc_{\hat{P}} (\phi_{0,k})} \left( \min_{P \in \mc{P}_k^1, 1\leq i \leq 4} L^{(k,k+n)}(R_i^S(P)) \right) \sum_{P \in \pi_{n+k}^k} e^{\xi \phi_{0,k}(P)},
$$
where $\pi_{n+k}^k$ denotes the $k$-coarse grained approximation of $\pi_{n+k}$,  the left-right geodesic of $[0,1]^2$ for the field $\phi_{0,n+k}$, and where we recall that $(R_{i}^S(P))_{1 \leq i \leq 4}$ denote the four rectangles of size $2^{-k}(1,3)$ surrounding $P$. Furthermore, by using a similar argument to \eqref{eq:CoarseToPath}, we have
$$
\sum_{P \in \pi_{n+k}^k} e^{\xi \phi_{0,k}(P)} \geq e^{-\xi \max_{P \in \mc{P}_k} \osc_P (\phi_{0,k})} 2^k L_{1,1}^{(k)}.
$$
Altogether, we get the following weak supermultiplicativity, 
\begin{equation}
\label{eq:WeakSuperMul}
L_{1,1}^{(n+k)} \geq L_{1,1}^{(k)}  \left(  \min_{P \in \mc{P}_k^1, 1\leq i \leq 4} 2^k L^{(k,k+n)}(R_i^S(P))\right)  e^{-2 \xi \max_{P \in \mc{P}_k} \osc_{P}(\phi_{0,k})}
\end{equation}
When $\xi$ satisfies Condition (T), by scaling and the tail estimates \eqref{eq:LowerTailsPhi},, $\Pro(  \min_{P \in \mc{P}_k^1, 1\leq i \leq 4} 2^k L^{(k,k+n)}(R_i^S(P)) \geq \lambda_{n} e^{-C \sqrt{k}}) \geq 1 - e^{-ck}$. Furthermore, using the gradient estimates \eqref{eq:OscBoundTail}, we get $\Pro(2^{-k} \norme{ \con \phi_{0,k} }_{[0,1]^2} \geq C \sqrt{k}) \geq 1 - e^{-ck}$ for $C$ large enough. Therefore, with probability $\geq 1/2$, $L_{1,1}^{(n)} \leq e^{-C \sqrt{k}} \lambda_{n} \lambda_k$ hence the bound $\lambda_{n+k} \geq e^{-C \sqrt{k}} \lambda_k \lambda_{n}$.
\end{proof}

\subsection{Tightness of the log of the diameter}
\label{Sec:TightDiam}

\begin{Prop}
\label{pro:diam}
If $\gamma \in (0,2)$ and $\xi = \gamma/d_{\gamma}$ then $\left( \log \mathrm{Diam}\left([0,1]^2, \lambda_{n}^{-1} e^{\xi \phi_{0,n}} ds \right) \right)_{n \geq 0}$ is tight.
\end{Prop}
\begin{proof} 
\textit{Step 1:} Chaining. By a standard chaining argument, (see (6.1) in \cite{DF18} for more details), we have  
\begin{equation}
\label{eq:Chaining}
\mathrm{Diam} \left( [0,1]^2 , e^{\xi \phi_{0,n}} ds^2 \right) \leq C \sum_{k=0}^n  \underset{P \in \mc{C}_k}{\max} \ L^{(n)}(P) + C \times 2^{-n} e^{\xi \sup_{[0,1]^2} \phi_{0,n}},
\end{equation}
where $\mc{C}_k$ is a collection of no more than $C 4^k$ long rectangles of side length $2^{-k}(3,1)$. 

Using the bound for the maximum \eqref{eq:MaxBoundExp}, when $\xi < 2$, we have $\E (2^{-n} e^{\xi \sup_{[0,1]^2} \phi_{0,n}}) \leq 2^{-n} 2^{2 \xi n} e^{C \sqrt{n}}$.

Fix $ 0 \leq k \leq n$ and $P \in \mc{C}_k$.  We can bound $L^{(n)}(P)$ by taking a left-right geodesic $\pi_{k,n}$ for $\phi_{k,n}$. Therefore, 
$$
L^{(n)}(P) \leq L^{(n)}(\pi_{k,n}) \leq e^{\xi \max_{[0,1]^2} \phi_{0,k}} L^{(k,n)}(P),
$$
and consequently, 
\begin{equation}
\label{eq:Decouplage}
\max_{P \in \mc{C}_k} L^{(n)}(P)  \leq e^{\xi \max_{[0,1]^2} \phi_{0,k}} \max_{P \in \mc{C}_k} L^{(k,n)}(P).
\end{equation}

Using independence, the maximum bound \eqref{eq:MaxBoundExp},  scaling of the field $\phi$ and the tail estimates \eqref{eq:UpperTailsPhi}, we get
\begin{equation}
\label{eq:inequa}
\E \left(e^{\xi \max_{[0,1]^2} \phi_{0,k}} \max_{P \in \mc{C}_k} L^{(k,n)}(P)\right) \leq 2^{-k} 2^{2\xi k} e^{C \sqrt{k}} \lambda_{n-k} e^{C k^{\frac{1}{2} + \eps}}
\end{equation}
for some fixed small $\eps >0$ (again, the term $k^{\eps}$ could in fact be $\log k$). Taking the expectation in \eqref{eq:Chaining}, using \eqref{eq:Decouplage} and \eqref{eq:inequa}, we obtain the following bound for the expected value of the diameter,
\begin{equation}
\label{eq:SumObtained}
\E ( \mathrm{Diam} ( [0,1]^2 , e^{\xi \phi_{0,n}} ds ) )  \leq C \sum_{k=0}^n 2^{-k} 2^{2 \xi k} \lambda_{n-k} e^{C k^{\frac{1}{2} + \eps}}.
\end{equation}

\textit{Step 2:} Right tail. By Proposition \ref{Prop:WeakMul}, $\lambda_{n-k} \leq \lambda_n \frac{e^{C \sqrt{k}} }{\lambda_k} \leq \lambda_n 2^{k(1-\xi Q)} e^{C \sqrt{k}}$. Together with \eqref{eq:SumObtained}, this implies that
$$
\E ( \mathrm{Diam} ( [0,1]^2 , e^{\xi \phi_{0,n}} ds ) ) \leq C \sum_{k=0}^n 2^{-k} 2^{2 \xi k} \lambda_{n-k} e^{C k^{\frac{1}{2} + \eps}} 
\leq \lambda_n C \sum_{k=0}^{\infty}  2^{-k \xi(Q-2)} e^{C k^{\frac{1}{2} + \eps}}.
$$
Since $Q  > 2$, Markov's inequality gives $\Pro \left( \mathrm{Diam} ( [0,1]^2 , \lambda_{n}^{-1} e^{\xi \phi_{0,n}} ds ) \geq e^s \right) \leq C e^{-s}$.

\textit{Step 3:} Left tail. Finally, since the diameter of the square $[0,1]^2$ is larger than the left-right distance, by our tail estimates \eqref{eq:LowerTailsPhi}, we get $\Pro \left( \mathrm{Diam} ( [0,1]^2 , \lambda_{n}^{-1} e^{\xi \phi_{0,n}} ds )  \leq e^{-s}  \right) \leq \Pro \left( L_{1,1}^{(n)} \leq \lambda_n e^{-s} \right) \leq C e^{-cs^2}$.
\end{proof}

\subsection{Tightness of the metrics}

\label{Sec:TightMetric}

\begin{Prop}
\label{Prop:Holder}
If $\gamma \in (0,2)$ and $\xi = \gamma/d_{\gamma}$ then the sequence of metrics $\left( \lambda_n^{-1} e^{\xi \phi_{0,n}} ds \right)_{n \geq 0}$ is tight. Moreover, if we define 
$$
C_{\alpha}^{n} := \sup_{x,x' \in [0,1]^2} \frac{|x-x'|^{\alpha}}{d_{0,n}(x,x')} \quad \text{and}  \quad
C_{\beta}^{n} :=  \sup_{x,x' \in [0,1]^2} \frac{d_{0,n}(x,x')}{|x-x'|^{\beta}}
$$
then, for $\alpha > \xi(Q+2)$ and $\beta < \xi (Q-2)$, the sequence $(C_{\alpha}^n, C_{\beta}^n)_{n \geq 0}$ is tight.
\end{Prop}
Henceforth, we use the notation $d_{0,n}$ for the renormalized metric $\lambda_n^{-1} e^{\xi \phi_{0,n}} ds$ restricted to $[0,1]^2$.
\begin{proof}

The proof has  two parts. In the first part we show the tightness of the metrics in the space of continuous function from $[0,1]^2 \times [0,1]^2 \to \R^{+}$ and in the second part we show that subsequential limits are metrics. 
A byproduct result of the argument is explicit bi-H\"older bounds. 

\textbf{Part 1.} Upper bound on the modulus of continuity. We suppose $\gamma \in (0,2)$. We start by proving that for every  $0 < \beta <  \xi (Q-2)$,  if $\eps > 0$, there exists a large $C_{\eps} > 0$ so that for every $n \geq 0$
\begin{equation}
\label{UpperHolder}
\Pro \left( \exists x,x' \in [0,1]^2 ~ : ~ d_{0,n}(x,x') \geq C_{\eps} |x-x'|^{\beta} \right) \leq \eps,
\end{equation} 
i.e. $\left( \norme{ d_{0,n} }_{C^{\beta}([0,1]^2 \times [0,1]^2)}\right)_{n \geq 0}$ is tight, where the $C^{\beta}$-norm is defined for $f : [0,1]^2 \times [0,1]^2 \to \R$ as
$$
\norme{f}_{C^{\beta}([0,1]^2 \times [0,1]^2)} := \norme{f}_{[0,1]^2 \times [0,1]^2} + \sup_{(x,y) \neq (x',y') \in [0,1]^2 \times [0,1]^2} \frac{ | f(x,y) - f(x',y')|}{|(x,y)-(x',y')|^{\beta}}.
$$

By a union bound it suffices to estimate $\Pro (  \exists x,x' : \  |x-x'|  <  2^{-n}  ,   d_{0,n}(x,x') \geq  e^{s} |x-x'|^{\beta} ) $ and
$$ \sum_{k=0}^{n} \Pro \left( \exists x,x' : 2^{-k} \leq |x-x'| \leq 2^{-k+1},  d_{0,n}(x,x') \geq e^{s} |x-x'|^{\beta}  \right).
$$

\textit{Step 1:} We start with the term $\Pro ( \exists x,x' :  2^{-k} \leq |x-x'| \leq 2^{-k+1},  d_{0,n}(x,x') \geq e^{s} |x-x'|^{\beta} )$.  We use the chaining argument \eqref{eq:Chaining} at scale $k$ which gives
$$
\sup_{2^{-k} \leq |x-x'| \leq 2^{-k+1}} d_{0,n}(x,x') \leq C \lambda_n^{-1} \sum_{i=k}^n  \underset{P \in \mc{C}_i}{\max }\ L^{(n)}(P) + C \lambda_n^{-1} \times 2^{-n} e^{\xi \underset{[0,1]^2}{\sup} \phi_{0,n}}.
$$
Taking the expected value and using the same bounds as those obtained in the proof of Proposition \ref{pro:diam}, we get
$$
\E \left( \sup_{2^{-k} \leq |x-x'| \leq 2^{-k+1}} d_{0,n}(x,x') \right) \leq  \sum_{i = k}^{n} 2^{-i \xi (Q-2)} e^{C i^{\frac{1}{2}+\eps}} \leq C 2^{-k \xi (Q - 2)} e^{C k^{\frac{1}{2} + \eps}}.
$$
Therefore, using Markov's inequality we get the bound
\begin{multline*}
\sum_{k=0}^{n} \Pro \left( \exists x,x' : 2^{-k} \leq |x-x'|  \leq 2^{-k+1},  d_{0,n}(x,x') \geq e^{s} |x-x'|^{\beta}  \right) \\
 \leq \sum_{k = 0}^n \Pro \left( \sup_{2^{-k} \leq |x-x'| \leq 2^{-k+1}} d_{0,n}(x,x')  \geq e^s 2^{-k \beta} \right) \leq e^{-s}  \sum_{k=0}^{n} 2^{k \beta} 2^{-k \xi (Q - 2)}.
\end{multline*}
The series is convergent since $\xi(Q-2) - \beta > 0$. 

\textit{Step 2:} We bound from above $\Pro (  \exists x,x' \  |x-x'|  <  2^{-n}  ,   d_{0,n}(x,x') \geq  e^{s} |x-x'|^{\beta} )$ using a bound on the supremum of the field. Indeed, for such $x$ and $x'$, note that
$$
e^{s} |x-x'|^{\beta} \leq d_{0,n}(x,x') \leq \lambda_n^{-1} e^{\xi \sup_{[0,1]^2} \phi_{0,n}} |x-x'|
$$
Writing $\beta = \xi(Q-2) -\eps \xi$ for some $\eps > 0$, it follows that $1-\beta = (1-\xi Q + 2\xi) + \eps \xi >0$ since the LFPP exponent $1-\xi Q \geq -2\xi$ by a simple  uniform bound. Therefore, $|x-x'|^{\beta-1} \geq 2^{n(1-\beta)}$ and $\lambda_n^{-1} 2^{n(1-\beta)} = 2^{n(2\xi + \eps \xi + o(1))}$. Altogether, this probability is bounded from above by $\Pro ( \sup_{[0,1]^2} \phi_{0,n} \geq n \log 4 + \eps n \log 2 + o(n) + \xi^{-1} s )$ and using \eqref{eq:MaxBoundTail} gives a uniform tail estimate.

Therefore, we obtain the tightness of $\left( d_{0,n} \right)_{n \geq 0}$ as a random element of $C([0,1]^2 \times [0,1]^2, \R^{+})$ and every subsequential limit is (by Skorohod's representation theorem) a pseudo-metric. 

\textbf{Part 2.} Lower bound on the modulus of continuity. We prove that if $\alpha>  \xi(Q+2)$ and  $\eps > 0$ then there exists a small constant $c_{\eps}>0$ such that for every $n \geq 0$,
\begin{equation}
\label{LowerHolder}
\Pro \left( \exists x,x' \in [0,1]^2 ~ : ~ d_{0,n}(x,x') \leq c_{\eps} |x-x'|^{\alpha} \right) \leq \eps.
\end{equation}
Similarly as before, by union bound it is enough to estimate the term 
\begin{equation}
\label{eq:FstPart}
\Pro(  \exists x,x' \in [0,1]^2 :  |x-x'|  <  2^{-n}  ,   d_{0,n}(x,x') \leq e^{-\xi s} |x-x'|^{\alpha}  )
\end{equation} 
and the term
\begin{equation}
\label{eq:SndPart}
\sum_{k=0}^{n} \Pro \left(\underbrace{ \exists x,x' : 2^{-k} \leq |x-x'| \leq 2^{-k+1},  d_{0,n}(x,x') \leq e^{-\xi s} |x-x'|^{\alpha}}_{:=E_{k,n,s}}  \right).
\end{equation}

\textit{Step 1:} We give an upper bound for \eqref{eq:SndPart}. Fix $ x,x' \in [0,1]^2$ such that $2^{-k} \leq |x-x'| \leq 2^{-k+1}$. Note that any path from $x$ to $x'$ crosses one of the rectangles in the collection $\lbrace R_i^S(P) : P \in \mc{P}_{k+2}^1, 1\leq i \leq 4 \rbrace$. Hence, under the event $E_{k,n,s}$, there exists $x, x'$ such that
\begin{equation}
\label{eq:BoundHolderLowR}
2^{-k \alpha} \geq d_{0,n}(x,x') \geq \lambda_n^{-1} 2^{-k} e^{\xi \inf_{[0,1]^2} \phi_{0,k}} \left(\min_{P \in \mc{P}_{k+2}^1, 1\leq i \leq 4}  2^k L^{(k,n)}(R_i^S(P))\right).
\end{equation}

Since $\alpha = \xi (Q+2) + \xi \delta$ for a small $\delta >0$, by using Proposition \ref{Prop:WeakMul} we get
\begin{equation}
\label{eq:BoundHolderLow}
2^{-k \alpha} \lambda_n 2^k \leq 2^{-k \alpha} \lambda_{k} \lambda_{n-k} e^{C \sqrt{k}} \leq 2^{-k(\alpha- \xi Q)}  \lambda_{n-k}  e^{C \sqrt{k}} = 2^{- k (2+\delta) \xi} (\lambda_{n-k} 2^{- \xi \delta k} e^{C \sqrt{k}})
\end{equation}
Now, using \eqref{eq:BoundHolderLowR}, \eqref{eq:BoundHolderLow} and scaling, we get
\begin{align*}
\Pro \left( E_{k,n,s} \right) & \leq \Pro \left ( e^{\xi \inf_{[0,1]^2} \phi_{0,k}}  \left( \min_{P \in \mc{P}_{k+2}^1, 1\leq i \leq 4}  2^k L^{(k,n)}(R_i^S(P))\right) \leq 2^{-k \alpha} \lambda_n 2^k   e^{-\xi s}   \right) \\
& \leq \Pro \left(  \sup_{[0,1]^2} |\phi_{0,k}| \geq k \log 4 + k \delta \log 2  + s/2  \right)  + \Pro \left (\min_{P \in \mc{P}_{k+2}^1, 1\leq i \leq 4}  L^{(n-k)}(R_i^S(P)) \leq \lambda_{n-k}   2^{-k \delta \xi} e^{C \sqrt{k}} e^{-\xi s/2}   \right)  \\
& \leq C e^{-ck} e^{-cs},
\end{align*}
where we used in the last inequality the supremum bounds \eqref{eq:MaxBoundTail} and the left tail estimate \eqref{eq:LowerTailsPhi}.

\textit{Step 2:} Finally, we control \eqref{eq:FstPart}.  We write
\begin{align*}
\Pro(  \exists x,x' :   |x-x'|  <  2^{-n}  ,   d_{0,n}(x,x') \leq e^{-\xi s} |x-x'|^{\alpha}) & \leq  \Pro \left( \inf_{|x-x'| \leq 2^{-n}} \frac{d_{0,n}(x,x')}{|x-x'|^\alpha} \leq e^{-\xi s} \right) \\
& \leq \Pro \left( \lambda_n^{-1} e^{\xi \inf_{[0,1]^2} \phi_{0,n}}  \inf_{|x-x'| \leq 2^{-n}} |x-x'|^{1-\alpha} \leq e^{-\xi s} \right).
\end{align*}
We recall that $\alpha > \xi Q + 2\xi$, and in particular $\alpha > 1$: indeed, $ 1-\xi Q \leq 2 \xi$ follows from a  comparison with the infimum of the field. In this case, $\inf_{|x-x'| \leq 2^{-n}} |x-x'|^{1-\alpha} = 2^{-n(1-\alpha)},$ and by Proposition \ref{Prop:WeakMul},
$$
2^{-n(1-\alpha)} \lambda_n^{-1} \geq 2^{-n(1-\alpha)} 2^{n(1 - \xi Q)} e^{-C \sqrt{n}} = 2^{n (\alpha-\xi Q)} e^{-C \sqrt{n}}
$$
Therefore, since $\alpha - \xi Q = 2 \xi + \delta \xi$ for some $\delta > 0$, we have for $n$ large that
\begin{align*}
\Pro \left( \lambda_n^{-1} e^{\xi \inf_{[0,1]^2} \phi_{0,n}}  \inf_{|x-x'| \leq 2^{-n}} |x-x'|^{1-\alpha} \leq e^{-\xi s} \right)  \leq \Pro \left( \sup_{[0,1]^2} | \phi_{0,n} | \geq n \log 4 + n  \frac{\delta}{2} \log 2  + s \right)
\end{align*}
Using \eqref{eq:MaxBoundTail} completes the proof. 
\end{proof}

\section{Appendix}

\subsection{Comparison with the GFF mollified by the heat kernel}

\label{Sec:HeatKernel}

Let $h$ be a GFF with Dirichlet boundary condition on a domain $D$ and $U \subset \subset D$ be a subdomain of $D$. We recall that we denote by $p_{t}$ the two-dimensional heat kernel at time $t$ i.e. $p_t(x) = \frac{1}{2\pi t} e^{-\frac{|x|^2}{2t}}$.  The goal of this section is to obtain a uniform estimate to conclude on the tightness of the renormalized metric associated to $p_{\frac{t}{2}} \ast h$ assuming the one associated to $\phi_{\sqrt{t}}$. In particular, the second assertion of Theorem \ref{thm:MainTheorem} is a corollary of the following proposition. 

\begin{Prop}
\label{Prop:GffHT} There exist constants $C, c > 0$ such that for all $t \in (0,1/2)$, there is a coupling of $h$ and $\varphi_t \overset{(d)}{=} \phi_{\sqrt{t}} $ such that  for all $x \geq 0$, we have
$$
\Pro \left( \norme{ \varphi_t - p_{\frac{t}{2}} \ast h }_{U} \geq x \right) \leq C e^{-cx^2}.
$$
\end{Prop}

\paragraph{Mollification of the GFF by the heat kernel.} 
The covariance of the Gaussian field $p_{\frac{t}{2}} \ast h$ is given for $x, x' \in U$ by
$$
\E \left( p_{\frac{t}{2}} * h(x)  \ p_{\frac{t}{2}} * h(x') \right) = \int_{D} \int_D p_{\frac{t}{2}}(x-y) G_D(y,y') p_{\frac{t}{2}}(y'-x') dy dy',
$$
where $G_D$ is the Green function associated to the Laplacian operator on $D$. For an open set $A$, we denote by $p_{t}^A(x,y)$ the transition probability density of a Brownian motion killed upon exiting $A$. 

\paragraph{White noise representation.} Take a space-time white noise $W$ and define the field $\eta_t$ on $U$ by
\begin{equation}
\label{Def:eta}
\eta_t(x) := \int_0^{\infty} \int_{D} p_{\frac{t}{2}} * p_{\frac{s}{2}}^D(x,y) W(dy,ds) \quad \text{where} \quad p_{\frac{t}{2}} * p_{\frac{s}{2}}^D(x,y) := \int_{D} p_{\frac{t}{2}}(x-y') p_{\frac{s}{2}}^D(y',y) dy',
\end{equation}
so that $( \eta_t(x) )_{x \in U} \overset{(d)}{=} (p_{\frac{t}{2}} * h (x))_{x \in U}$. Indeed, by Fubini, we have
\begin{align*}
\E(\eta_t(x) \eta_t(x')) & =  \int_0^{\infty} \int_{D} p_{\frac{t}{2}} * p_{\frac{s}{2}}^D(x,y) \  p_{\frac{t}{2}} * p_{\frac{s}{2}}^D(x',y) dy ds \\
&  = \int_{0}^{\infty} \int_D \int_D \int_D p_{\frac{t}{2}}(x-y') p_{\frac{s}{2}}^D(y',y) \ p_{\frac{t}{2}}(x'-y'') \ p_{\frac{s}{2}}^D(y'',y) dy dy' dy'' ds \\
& = \int_D \int_D  p_{\frac{t}{2}}(x-y')  \left( \int_{0}^{\infty}  \int_D p_{\frac{s}{2}}^D(y',y) p_{\frac{s}{2}}^D(y,y'') dy ds \right) p_{\frac{t}{2}}(x'-y'') dy' dy'' \\
& = \int_{D} \int_D p_{\frac{t}{2}}(x-y') G_D(y',y'') p_{\frac{t}{2}}(y''-x') dy' dy''.
\end{align*}

\paragraph{Coupling.} Note that for $t \in (0,1/2)$ $\phi_{\sqrt{t}}(x) = \int_{t}^{1} \int_{\R^2} p_{\frac{s}{2}}(x-y) W(dy,ds)  \overset{(d)}{=} \varphi_t(x)$, where we set
\begin{align*}
\varphi_t(x) := \int_{0}^{1-t} \int_{\R^2}  p_{\frac{t+s}{2}}(x-y) W(dy,ds).
\end{align*}
Furthermore, we can decompose $\varphi_t(x)  = \varphi_t^1(x) + \varphi_t^2(x)$, where
\begin{align}
\label{Def:phi1}
\varphi_t^1(x) :=  \int_{0}^{1-t} \int_{D} p_{\frac{t+s}{2}}(x-y) W(dy,ds); \\
\varphi_t^2(x) := \int_{0}^{1-t} \int_{D^c} p_{\frac{t+s}{2}}(x-y) W(dy,ds).
\end{align}
Recalling the definition of $\eta$ in \eqref{Def:eta}, we introduce $\eta_t^1$ and $\eta_t^2$ so that
\begin{equation}
\label{DefEtai}
\eta_t(x)  =  \int_0^{1-t} \int_{D} p_{\frac{t}{2}} * p_{\frac{s}{2}}^D(x,y) W(dy,ds) +  \int_{1-t}^{\infty} \int_{D} p_{\frac{t}{2}} * p_{\frac{s}{2}}^D(x,y) W(dy,ds) =: \eta_t^1(x) + \eta_t^2(x).
\end{equation}
Therefore, under this coupling (viz. using the same white noise $W$), we have
\begin{equation}
\label{eq:Coupling}
\varphi_t^1(x) - \eta_t^1(x) = \int_0^{1-t} \int_D \left( p_{\frac{t+s}{2}}(x-y) - p_{\frac{t}{2}} * p_{\frac{s}{2}}^D(x,y) \right) W(dy,ds).
\end{equation}

\paragraph{Comparison between kernels.} We will consider $x,y \in U$, subdomain of $D$. Set $d := d(U,D^c) >0$.
$$
p_{\frac{t}{2}} * p_{\frac{s}{2}}^D(x,y) := \int_{D} p_{\frac{t}{2}}(x-y') p_{\frac{s}{2}}^D(y',y) dy' = \int_{D} p_{\frac{t}{2}}(x-y') p_{\frac{s}{2}}(y'-y) q_{\frac{s}{2}}^D(y',y)  dy',
$$
where $q_t^D(x,x')$ is the probability that a Brownian bridge between $x$ and $x'$ with lifetime $t$ stays in $D$. Therefore, using Chapman-Kolmogorov,
$$
p_{\frac{t}{2}} * p_{\frac{s}{2}}^D(x,y) - p_{\frac{t+s}{2}}(x,y) =  - \int_{D^c} p_{\frac{t}{2}}(x-y') p_{\frac{s}{2}}(y'-y)  dy' + \int_{D} p_{\frac{t}{2}}(x-y') p_{\frac{s}{2}}(y'-y) (q_{\frac{s}{2}}^D(y',y) - 1)  dy'.
$$
Note that the first term can be bounded by using that $|y-y'| \geq d$ for $y \in U$ and $y' \in D^c$. For the second term, we can split the integral over $D$ in two parts: one over the $\eps$-neighborhood of $\partial D$  (within $D$), denoted by $(\partial D)^{\eps}$, and one over its complement. To give an upper bound on the first, we use that for $y \in U$ and $y' \in (\partial D)^{\eps}$, $| y - y' | \geq d(U, (\partial D)^{\eps})$. Finally, we bound the second part by using a uniform estimate on the probability that a Brownian bridge between a point in $U$ and a point $D \setminus (\partial D)^{\eps}$ exits $D$ in time less than $s/2$. (Note that $ 1- q_{\frac{s}{2}}^D(y,y')$ is the probability that a Brownian bridge between $y$ and $y'$ with time length $s/2$ exits $D$.) Therefore, we get that uniformly in $x,y \in U$ and $t$, 
\begin{equation}
\label{eq:KernelComp}
| p_{\frac{t}{2}} * p_{\frac{s}{2}}^D(x,y) - p_{\frac{t+s}{2}}(x,y) | \leq  C e^{-\frac{c}{s}}.
\end{equation}

\paragraph{Comparison between $\varphi_{t}$ and $p_{\frac{t}{2}} * h$.} By the triangle inequality,
\begin{equation}
\label{eq:ComparionHeat}
\norme{ \varphi_t - p_{\frac{t}{2}} \ast h }_{U} \leq \norme{ \varphi_t^1 - \eta_t^1 }_{U} + \norme{ \varphi_t^2 }_{U} +  \norme{ \eta_t^2 }_{U}.
\end{equation}
We look for a uniform  right tail estimate (in $t$) of each term in the right-hand side of \eqref{eq:ComparionHeat}. In order to do so, we will use the Kolmogorov continuity criterion. Therefore, we derive below some pointwise and difference estimates.

\textbf{First term.} We derive first a pointwise estimate. For $x \in U$, using the kernel comparison \eqref{eq:KernelComp}, there exists some $C' > 0$ such that, uniformly in $t$,
$$
\Var \left( \left(\eta_{t}^1(x) - \varphi_{t}^2(x) \right)^2  \right) = \int_0^{1-t} \int_{D} \left( p_{\frac{t}{2}} * p_{\frac{s}{2}}^D(x,y) -  p_{\frac{t+s}{2}}(x,y)  \right)^2 dy ds \leq C \int_{0}^{1-t} e^{- \frac{c}{s}} ds \leq C'.
$$
We now give a difference estimate: introducing $\Delta_t(x) := \varphi_t^1(x) - \eta_t^1(x)$, for $x,x' \in U$,
$$
\E \left( \left( \Delta_t(x) -\Delta_t(x') \right)^2 \right) = \int_0^{1-t} \int_D \left( \left( p_{\frac{t+s}{2}}(x-y) - p_{\frac{t}{2}} * p_{\frac{s}{2}}^D(x,y) \right)  -  \left( p_{\frac{t+s}{2}}(x'-y) - p_{\frac{t}{2}} * p_{\frac{s}{2}}^D(x',y) \right)  \right)^2 dy ds,
$$
which is uniformly bounded in $t \in (0,1/2)$ by a quantity of size $O(|x-x'|)$. (By splitting the integral at $\sqrt{|x-x'|}$, one can use \eqref{eq:KernelComp} for the small values of $s$ and gradient estimates for both kernels for larger values of $s$.)

\textbf{Second term.} We recall here that $\varphi_t^2(x)$ is defined for $x \in U$ by
$$
\varphi_t^2(x) = \int_{0}^{1-t} \int_{D^c} p_{\frac{t+s}{2}}(x-y) W(dy,ds) \overset{(d)}{=} \int_{t}^{1} \int_{D^c} p_{\frac{s}{2}}(x-y) W(dy,ds).
$$
We have, for $x,x' \in U$, with $d \coloneqq d(U,D^c)$, 
\begin{align*}
\E &\left( \left( \varphi_t^2(x) - \varphi_t^2(x') \right)^2 \right)  \leq \int_{t}^1 \int_{D^c} \left( p_{\frac{s}{2}}(x-y) - p_{\frac{s}{2}}(x'-y) \right)^2 dy ds \\
&  \leq \int_{\sqrt{|x-x'|}}^1 \int_{\R^2 } \left( p_{\frac{s}{2}}(x-y) - p_{\frac{s}{2}}(x'-y) \right)^2 dy ds  + \int_{0}^{\sqrt{|x-x'|}} \int_{D^c} \left( p_{\frac{s}{2}}(x-y) - p_{\frac{s}{2}}(x'-y) \right)^2 dy ds \\
& \leq  2\int_{\sqrt{|x-x'|}}^1   \left(p_{s}(0) - p_{s}(x-x')\right) ds + 4 \int_{0}^{\sqrt{|x-x'|}}p_{\frac{s}{2}}(d) ds \leq C |x-x'|,
\end{align*}
where we use $1-e^{-z} \leq z$ in the last inequality. Similarly, we can prove that there exists $C >0$ independent of $t$ such that $\E(\phi_t(x)^2) \leq C$. 

\textbf{Third term.} We recall here that $\eta_t^2(x)$ is defined for $x \in U$ by $\eta_t^2(x) =  \int_{1-t}^{\infty} \int_{D} p_{\frac{t}{2}} * p_{\frac{s}{2}}^D(x,y) W(dy,ds)$. Similarly, there exists $C >0$ such that for $t \in (0,1/2)$, $x,x' \in U$, we have 
$$
\E \left( \left( \eta_t^2(x)  - \eta_t^2(x')\right)^2  \right) \leq \int_{1/2}^{\infty} \int_D \left(p_{\frac{t}{2}} * p_{\frac{s}{2}}^D(x,y) - p_{\frac{t}{2}} * p_{\frac{s}{2}}^D(x',y)  \right)^2 dy ds \leq C |x-x'|.
$$
Furthermore, the pointwise variance is uniformly bounded.

\textbf{Result.} Altogether, coming back to \eqref{eq:ComparionHeat} and combining Kolmogorov continuity criterion with Fernique's theorem (see Section 1.3 in \cite{Fernique}), we get the following tail estimate on the above coupling: there exist $C,c > 0$ such that for all $t \in (0,1/2)$, $x \geq 0$, we have
$$
\Pro \left( \norme{ \varphi_t - p_{\frac{t}{2}} \ast h }_{U} \geq x \right) \leq C e^{-cx^2}.
$$

\subsection{Approximations for $\delta \in (0,1)$}

We explain here how results obtained along the sequence $\lbrace 2^{-n} : n \geq 0 \rbrace$ can be extended to $\delta \in (0,1)$. For each $\delta \in (0,1)$, let $n \geq 0$ and $r \in [0,1]$ such that $\delta = 2^{-(n+r)}$. Then by decoupling the field $\phi_{0,r}$, using a uniform estimate for $r\in [0,1]$ and a scaling argument, we generalize our previous results obtained along the sequence $2^{-n}$ to $\delta \in (0,1)$.

\paragraph{Decoupling low frequency noise.} Note that there exists $C >0$ such that for $n \geq 0$ and $r \in [0,1]$ we have
\begin{equation}
\label{eq:AprioriMul}
e^{-C} \lambda_n \leq  \lambda_{n+r} \leq \lambda_n e^{C}.
\end{equation}
Indeed, note that a.s. $e^{-\xi \inf_{[0,1]^2}  \phi_{0,r}} L_{1,1}^{(r,n+r)} \leq L_{1,1}^{(n+r)} \leq e^{\xi \sup_{[0,1]^2} \phi_{0,r}} L_{1,1}^{(r,n+r)}$. Furthermore, with high probability $\sup_{[0,1]^2} | \phi_{0,r} | \leq C_r \leq C$. Then, note that $ L_{1,1}^{(r,n+r)} \overset{(d)}{=} 2^{-r} L_{2^{r},2^r}^{(n)}$ and a.s.  $L_{1,2}^{(n)} \leq L_{2^r,2^r}^{(n)} \leq L_{2,1}^{(n)}$. By the tightness result, there exists a constant $C > 0$ such that uniformly in $n$, with high probability, $L_{1,2}^{(n)} \geq e^{-C} \lambda_n$ and $L_{2,1}^{(n)} \leq e^{C} \lambda_n$, therefore, with high probability, $ e^{-C} \lambda_n  \leq L_{1,1}^{(r,n+r)}  \leq e^{C} \lambda_n$, hence \eqref{eq:AprioriMul}.

\paragraph{Weak multiplicativity.}

In this paragraph, we will use the notation $\lambda_{\delta}$ from the introduction. We recall that writing $\lambda_n$ instead of $\lambda_{2^{-n}}$ was an abuse of notation. Now we prove that there exists $C > 0$ such that for $\delta, \delta' \in (0,1)$ we have
\begin{equation}
\label{eq:MulCont}
C^{-1} e^{-C \sqrt{ | \log \delta \vee \delta' |}} \lambda_{\delta} \lambda_{\delta'}  \leq \lambda_{\delta \delta'} \leq  C e^{C \sqrt{ | \log \delta \vee \delta' |}}  \lambda_{\delta} \lambda_{\delta'}.
\end{equation}
Similarly as \eqref{eq:AprioriMul}, there exists $C > 0$ such that for $r, r' \in [0,1]$, $n, n' \geq 0$,
\begin{equation}
\label{eq:S1}
e^{-C} \lambda_{2^{-n-n'}} \leq \lambda_{2^{-n - r - n' - r'}} \leq \lambda_{2^{-n-n'}} e^C.
\end{equation}
For $\delta, \delta' \in (0,1)$, let $n, n' \geq 0$ and $r, r' \in [0,1]$ such that $\delta = 2^{-(n+r)}$, $\delta' = 2^{-(n'+r')}$. Note that $n = [ -\log_2 \delta]$. Using the weak multiplicativity for powers of $2$, we have
\begin{equation}
\label{eq:S2}
e^{- C \sqrt{n \wedge n' }}  \lambda_{2^{-n}} \lambda_{2^{-n'}} \leq \lambda_{2^{-n-n'}} \leq  \lambda_{2^{-n}} \lambda_{2^{-n'}} e^{C \sqrt{n \wedge n' }}.
\end{equation}
Without loss of generality, we consider just the upper bound in \eqref{eq:MulCont}. The lower bound follows along the same lines. By using first \eqref{eq:S1} and then \eqref{eq:S2} we get
$$
\lambda_{\delta \delta'} = \lambda_{2^{-n-r - n' -r'}} \leq \lambda_{2^{-n-n'}} e^C \leq \lambda_{2^{-n}} \lambda_{2^{-n'}} e^{C \sqrt{n \wedge n'}} e^C.
$$
Now, the result follows by using \eqref{eq:AprioriMul}:
$$
\lambda_{2^{-n}} \lambda_{2^{-n'}} e^{C \sqrt{n \wedge n'}} \leq \lambda_{2^{-n-r}} \lambda_{2^{-n'-r'}}  e^{C \sqrt{n+r \wedge n'+r'}} e^{2C} = \lambda_{\delta} \lambda_{\delta'} e^{C \sqrt{\log | \delta \vee \delta'|}} e^{2C}.
$$

\paragraph{Tail estimates and tightness of metrics.} Using the same argument as in the two previous paragraphs and the tail estimates obtained along the sequence $\lbrace 2^{-n} : n \geq 1 \rbrace$, we have the following tail estimates for crossing lengths of the rectangles $[0,a] \times [0,b]$: there exists $c, C >0$ (depending only on $a$, $b$ and $\gamma$) such that for $s > 2$, uniformly in $\delta \in (0,1)$, we have
\begin{align}
\label{eq:RightTails}
\Pro \left( \lambda_{\delta}^{-1} L_{a,b}^{(\delta)} \geq e^{s} \right) \leq C e^{-c \frac{s^2}{\log s}}; \\
\label{eq:LeftTails}
\Pro \left( \lambda_{\delta}^{-1} L_{a,b}^{(\delta)} \leq e^{-s} \right) \leq C e^{-c s^2}.
\end{align}
Furthermore, the sequence of metrics $( \lambda_{\delta}^{-1} e^{\xi \phi_{\delta}} ds)_{\delta \in (0,1)}$ on $[0,1]^2$ is tight.

\paragraph{Acknowledgments.} We would like to thank the referees for many helpful comments which helped to improve the exposition of the manuscript.

\bibliographystyle{habbrv}
\bibliography{biblio}

\end{document}